\documentclass[12pt]{amsart}

\usepackage{upref,amssymb,bm,mathrsfs}
\usepackage{enumerate}
\usepackage{mathtools}
\usepackage{soul}

\usepackage[centering,width=6in,truedimen]{geometry}
\usepackage[colorlinks,pagebackref,pdfstartview=FitH]{hyperref}

\hypersetup{
    colorlinks,
    citecolor=blue,
    filecolor=black,
    linkcolor=blue,
    urlcolor=black
}

\theoremstyle{plain}
\newtheorem{thm}{Theorem}[section]
\newtheorem{prop}[thm]{Proposition}
\newtheorem{lem}[thm]{Lemma}
\newtheorem{cor}[thm]{Corollary}

\newtheorem{ques}[thm]{Question}
\newtheorem*{ques*}{Question}
\newtheorem{oques}[thm]{Open question}
\theoremstyle{definition}
\newtheorem{defn}[thm]{Definition}

\theoremstyle{remark}
\newtheorem{rem}[thm]{Remark}

\DeclareMathOperator{\hdim}{dim_H}
\DeclareMathOperator{\adim}{dim_A}
\DeclareMathOperator{\ldim}{dim_L}
\DeclareMathOperator{\bdim}{dim_B}
\DeclareMathOperator{\ubdim}{\overline{dim}_B}
\DeclareMathOperator{\lbdim}{\underline{dim}_B}
\DeclareMathOperator{\pdim}{dim_P}

\newcommand{\N}{\mathbb N}

\newcommand{\R}{\mathbb R}

\numberwithin{equation}{section}
\renewcommand{\epsilon}{\varepsilon}
\newcommand{\supp}{\mbox{supp}}

\newcommand{\tb}[1]{\textcolor{blue}{#1}}

\usepackage{float}

\title[Distinct dimensions for attractors of  iterated function systems]{Distinct dimensions for attractors of bi-Lipschitz iterated function systems}

\author{Simon Baker}

\address{
Mathematical Sciences\\
Loughborough University\\
Loughborough\\
LE11 3TU, UK
}
\email{simonbaker412@gmail.com}

\author{Amlan Banaji}

\address{
Department of Mathematics and Statistics \\
University of Jyv\"askyl\"a \\
P.O. Box 35 (MaD), FI-40014 \\
Finland
}
\email{banajimath@gmail.com}

\author{De-Jun Feng}
\address{
	Department of Mathematics\\
	The Chinese University of Hong Kong\\
	Shatin, Hong Kong
}
\email{djfeng@math.cuhk.edu.hk}
\author{Chun-Kit Lai}
\address{
	Department of Mathematics\\
	San Francisco State University\\
	1600 Holloway Avenue, San Francisco, CA 94132
}
\email{cklai@sfsu.edu}

\author{Ying Xiong}
\address{
	Department of Mathematics\\
	South China University of Technology\\
	Guangzhou 510641, Guangdong\\
	People's Republic of China
}
\email{xiongyng@gmail.com}



\begin{document}

\begin{abstract}
	In this paper, we construct an iterated function system on $\mathbb{R}$ consisting of two bi-Lipschitz contractions whose attractor has distinct lower, Hausdorff, lower box, upper box, and Assouad dimensions, thereby providing negative answers to certain folklore questions.

      Furthermore, as a by-product of our study of bi-Lipschitz IFSs, we construct IFSs within this family that exhibit interesting fractal behaviour. In particular, we prove the following two statements:
(i) There exists a bi-Lipschitz IFS for which the pushforward of any ergodic measure with positive entropy is not exact dimensional; (ii)
 There exists a bi-Lipschitz IFS whose attractor has empty interior yet positive Lebesgue measure.

\end{abstract}

\maketitle

\section{Introduction}

\subsection{Background}

By an iterated function system (IFS) on $\R^d$, we mean a finite collection $\Phi=\{\phi_i\}_{i=1}^\ell$ of contraction mappings from $\R^d$ to itself. It is well known \cite{Hutch81} that given an IFS $\Phi=\{\phi_i\}_{i=1}^\ell$ on $\R^d$, there exists a unique non-empty compact set $E\subset \R^d$, called the {\bf attractor} of $\Phi$, such that
\[
E  = \bigcup_{i=1}^{\ell} \phi_i(E).
\]
The attractor $E$ is said to be  {\bf self-similar} if all $\phi_i$ are similitudes, and  {\bf self-affine}  if all $\phi_i$ are affine maps. The IFS $\Phi$ is said to be a {\bf bi-Lipschitz IFS} if the maps $\phi_i$ are bi-Lipschitz contractions on $\R^d$, i.e., there exist constants $0<A\le B<1$ such that
\[
A|x-y|\le |\phi_i(x)-\phi_i(y)|\le B|x-y|, ~\mbox{ for all } x,y\in\R^d.
\]
 One of the fundamental problems in fractal geometry is to study the dimension theory of attractors of IFSs; see e.g.~\cite{BSS2023, BisPe17, Falconer2014, Mattila1995}. In this paper, we are mainly concerned with the following two folklore questions.

\begin{ques}
\label{q-1}
Does the box dimension of the attractor of an IFS on $\R^d$ always exist?
\end{ques}
\begin{ques}\label{ques:hausdorffbox}
    Do the Hausdorff and lower box dimensions of the attractor of an IFS on $\R$ always coincide?
\end{ques}

 It is well known that Hausdorff dimension, lower box dimension and upper box dimension satisfy $\dim_{\mathrm H} E \leq \lbdim E \leq \ubdim E$ for all non-empty bounded sets $E \subset \R^d$, and that these inequalities can be strict in general. We say that the box dimension of $E$ exists if $\lbdim E = \ubdim E$. The question of which classes of IFSs allow these inequalities to be strict is much more subtle. In \cite[Theorems~3-4]{Falco89} (see also \cite[Chapter~3]{Falco97}), Falconer gave some sufficient conditions for these three notions of dimension to coincide for a compact metric space. As an application, he showed that, without requiring any separation conditions, the box dimension of a self-similar set in $\R^d$ always exists and coincides with its Hausdorff dimension. Furthermore, this property (namely, the coincidence of the Hausdorff, lower box, and upper box dimensions) also holds for all $C^{1+\delta}$ conformal repellers. Later, this property was also verified for all $C^1$ conformal repellers by Gatzouras and Peres \cite{GatPe97} and Barreira \cite{Barreira1996}, for all $C^{1+\delta}$ random self-conformal sets by Liu and Wu \cite{LiuWu2003}, and for the attractors of all $C^1$ (weakly) conformal IFSs by Feng and Hu \cite[Theorem~8.1]{FenHu09}. However, the situation for non-conformal repellers or self-affine sets is quite different. It is well known \cite{Bedford1984, McMul84} that the Hausdorff and box dimensions of a special class of planar self-affine sets, known as  Bedford--McMullen carpets, are generally distinct. However, the box dimension of Bedford--McMullen carpets always exists, and it remains an important open question whether the box dimension of every self-affine set exists.

Recently, Jurga \cite{Jurga23} showed that there exist integers $n>m\geq 2$ and  a compact subset $E$ of the $2$-torus ${\mathbb T}^2$ such that $E$ is invariant under the expanding toral endomorphism $T(x,y) = (mx \pmod{1}, ny \pmod{1})$, and  the box dimension of $E$ does not exist. However, this invariant set may not be the attractor of an IFS. It has been known since the classical work of Mauldin and Urba\'nski \cite{MU1996,MU1999} that the attractor of an infinite conformal IFS (which is generally non-compact) can have distinct Hausdorff and upper box dimensions. Recently Banaji and Rutar~\cite{BanRu24} showed that for these sets, the Hausdorff dimension, lower box dimension, and upper box dimension can all be distinct. There are more detailed discussions about different notions of dimension for different classes of fractals in \cite{BSS2023, BisPe17}. In particular, in \cite[Theorem~2.8.4]{BisPe17}, it was shown that the packing dimension is always equal to the upper box dimension for the attractor of a bi-Lipschitz IFS.

\subsection{Attractors with distinct dimensions}

As the main result of this paper, we provide an explicit construction of an IFS consisting of two bi-Lipschitz contractions on $\R$, for which the box dimension of its attractor does not exist. Moreover, the lower, Hausdorff, lower box, upper box, and Assouad dimensions of the attractor are all distinct.   This gives a negative answer to both Question~\ref{q-1} and Question~\ref{ques:hausdorffbox}. The reader is referred to Section \ref{S-2} for definitions of various notions of dimension.

\begin{thm}\label{t:!bdim}
	There exists a bi-Lipschitz iterated function system on $\R$ such that the lower, Hausdorff, lower box, upper box and Assouad dimensions of the attractor are all distinct.
\end{thm}

\begin{rem}\label{rem:higherdims}
    For an integer $d \geq 2$, the product of the attractor from Theorem~\ref{t:!bdim} with the $(d-1)$-fold product of the middle-third Cantor set has distinct lower, Hausdorff, lower box, upper box and Assouad dimensions, and is the attractor of a bi-Lipschitz IFS on $\R^d$. Thus, Theorem~\ref{t:!bdim} also holds in higher dimensions.
\end{rem}

Theorem~\ref{t:!bdim} will be restated as Theorem~\ref{thm-5.1} in Section \ref{S-5}, where we provide a concretely constructed Cantor set ${\bf E}^*$, which is the attractor of a bi-Lipschitz IFS and has distinct dimensions; moreover the precise values of these dimensions are given in Theorem~\ref{thm-5.1}.

The idea of our construction of ${\bf E}^*$ can be summarized as follows: Given a Cantor set $E$ with $\{0,1\}\subset E\subset [0,1]$, and employing the standard Cantor construction of $E$, we are able to define two strictly increasing continuous functions, $\phi_0, \phi_1$, mapping from $[0,1]$ to itself, such that
\begin{equation}\label{ssc}
E=\phi_0(E)\cup \phi_1(E),
\end{equation}
with the union being disjoint. Furthermore, we prove a result of independent interest which provides
 a necessary and sufficient condition  for  $\phi_0, \phi_1$ to be bi-Lipschitz contraction maps on $[0,1]$;\footnote{In this case, we can extend $\phi_0, \phi_1$  to bi-Lipschitz contraction maps on $\R$ by simply letting $\phi_0, \phi_1$ be piecewise linear  on $(-\infty, 0] \cup [1,\infty)$ with slope $1/2$.} see Theorem~\ref{thm:bilip}.
Next, we will apply this theorem to a specially constructed  Cantor set ${\bf E}$. In this construction, the $n$-th order generating intervals of ${\bf E}$  have lengths given by $b_{\omega}M^{-n}$, where $M$ is a fixed large positive number,  and $b_{\omega}$ is an oscillatory function depending upon the coding $\omega$ of the generating interval. The precise formula for $b_{\omega}$ will be given in Section \ref{S-4};  the introduction of this special length function was inspired by the calculation of the dimensions of Bedford--McMullen carpets \cite{McMul84}.  We will show that the Cantor set ${\bf E}$ is generated by a bi-Lipschitz IFS consisting of two maps (see Proposition~\ref{p:BLIFS}) and has distinct Hausdorff and box dimensions (see Theorem~\ref{t1:dimE}). However, as shown in Theorem~\ref{t1:dimE}, the box dimension of ${\bf E}$ exists.  Finally, to separate the upper and lower box dimensions, we construct a new Cantor set ${\bf E}^*$, such that the $n$-th order generating intervals of ${\bf E}^*$  have lengths given by $b_{\omega}(\prod_{i=1}^nM_i)^{-1}$, where $(M_i)_{i=1}^\infty$ is a properly constructed sequence, taking only two values $M$ and $(M+1)$. Similarly, ${\bf E}^*$ is the attractor of  a bi-Lipschitz IFS consisting of two maps. Not only are the Hausdorff, upper box,  and  lower box dimensions distinct for ${\bf E}^*$, but the other two notions of dimension (the lower and Assouad dimensions) are also distinct from these three. It is worth pointing out that the corresponding IFSs of both ${\bf E}$ and ${\bf E}^*$ satisfy the strong separation condition, as stated in \eqref{ssc}.

It is possible to offer a simpler construction of attractors of bi-Lipschitz IFSs whose box dimension does not exist by using a class of symmetric Cantor sets. Due to the homogeneity of the length of the intervals in each iteration, we can obtain an easily-checkable sufficient condition for a symmetric Cantor set to be the attractor of a bi-Lipschitz IFS; see Proposition~\ref{prop:bilip1}. Meanwhile, it is well known that the box dimension of a symmetric Cantor set does not necessarily exist (see~\cite{FeWeW97}). Hence by Proposition~\ref{prop:bilip1} and the dimensional results in \cite{FeWeW97}, we can construct symmetric Cantor sets, which serve as the attractors of bi-Lipschitz IFSs, but their box dimension does not exist. However, in this construction, the Hausdorff dimension is always equal to the lower box dimension, so it does not offer a proof for Theorem~\ref{t:!bdim}. Nonetheless, as an interesting part of this simpler construction, we will prove that the contraction maps are non-differentiable at all points inside this Cantor set (see Proposition~\ref{prop-no-diff}).

\subsection{Further results}
While the main focus of this paper is the dimension theory of attractors of IFSs, as a by-product of our arguments we are able to construct bi-Lipschitz IFSs with other interesting properties.

\subsubsection{Non-exact dimensional measures}
\label{S-1.3.1}
Given a Borel probability measure $\mu$ on $\R^d$, we denote the lower and upper \textbf{local dimensions} of $\mu$ at a point $x$ in $\supp(\mu)$ by
\[
\underline{\dim}(\mu,x) = \liminf_{\delta \to 0} \frac{\log (\mu(B(x,\delta)))}{\log \delta}   , \qquad \overline{\dim}(\mu,x) = \liminf_{\delta \to 0} \frac{\log (\mu(B(x,\delta)))}{\log \delta},
\]
where $B(x,\delta)$ stands for the closed ball of radius $\delta$ centred at $x$.  If there exists a constant $c$ such that $\underline{\dim}(\mu,x) = \overline{\dim}(\mu,x) = c$ for $\mu$-a.e.~$x \in \supp(\mu)$, then $\mu$ is said to be \textbf{exact dimensional}.

Given an IFS $\{\phi_i\}_{i=1}^{\ell}$ of contraction maps on $\R^d$ and positive numbers $p_1,\dotsc,p_\ell > 0$ with $p_1 + \dotsb + p_{\ell} = 1$, it is well known \cite{Hutch81} that there exists a unique Borel probability measure $\mu$ on $\R^d$, called the \textbf{stationary measure} associated with
$\{\phi_i\}_{i=1}^{\ell}$ and $(p_i)_{i=1}^\ell$, satisfying
\[
\mu = \sum_{i=1}^{\ell} p_i (\phi_i)_* \mu,
\]
 where $f_* \mu$ denotes the push-forward of $\mu$  by $f$ (that is,  $f_* \mu(A) = \mu(f^{-1}(A))$ for all Borel sets $A\subset \R^d$).  The support $\supp(\mu)$ equals the attractor of the IFS.
The following question is natural:

\begin{ques}\label{ques:exactifs}
    Is every stationary measure associated with an IFS on $\R^d$ exact dimensional?
\end{ques}
Question~\ref{ques:exactifs} has been answered  affirmatively in the special cases of $C^1$ conformal IFSs~\cite{FenHu09} and affine IFSs~\cite{BaranyKaenmakiExact,FengExact}.
However, in what follows, we use a simple symmetric Cantor set construction to show that the answer is negative in general.
\begin{thm}\label{thm:notexact}
There exists an IFS consisting of two bi-Lipschitz contractions on $\R$ for which the $(1/2,1/2)$ stationary measure $\mu$ satisfies $$\underline{\dim}(\mu,x) = \lbdim E < \ubdim E = \overline{\dim}(\mu,x)$$ for all $x \in E$, where $E$ is the attractor of the IFS.  Hence, in particular, $\mu$ is not exact dimensional.
\end{thm}

Moreover, in Section~\ref{subsec:exactproof}, we will use the Shannon--McMillan--Breiman theorem to prove a more general result (see Theorem~\ref{thm:symmetricnotexact}) that there are symmetric Cantor sets on which the natural projection of every ergodic shift-invariant measure with positive entropy is not exact dimensional.
Similar to Remark~\ref{rem:higherdims}, our construction can be adapted to give stationary measures for IFSs on $\R^d$, $d \geq 2$, which are not exact dimensional.

\subsubsection{Attractors of IFSs with positive Lebesgue measure and empty interior}
In \cite{PerSol}, Peres and Solomyak posed the question of whether there exist self-similar sets with positive Lebesgue measure but empty interior. Cs\"{o}rnyei et al.~provided an affirmative  answer to this question in \cite{CJPPS} by demonstrating the existence of  such sets  in $\mathbb{R}^{2}$. The problem still remains open for the one-dimensional case, i.e., whether such sets exist in $\mathbb{R}$. In this paper, we establish an analogous result for attractors of bi-Lipschitz IFSs.

\begin{thm}
\label{thm:emptyinterior}
There exists an IFS consisting of two bi-Lipschitz contractions on $\R$ whose attractor has  positive Lebesgue measure but empty interior.
\end{thm}
  This result relates to the work of Bowen \cite{Bo75}, who constructed a $C^1$ horseshoe of positive Lebesgue measure for the first time.

\subsection{Dynamical interpretation}

There is a natural dynamical interpretation of our results. Corollary~\ref{cor:dynamical} below is a straightforward consequence of Theorem~\ref{t:!bdim}. It illustrates that the dimension theory of dynamical systems exhibits  significant differences when the expanding map has only Lipschitz regularity, as opposed to $C^1$ regularity. Indeed, if in Corollary~\ref{cor:dynamical} the map $f$ is further assumed to be $C^1$, then the dimension conclusions for the invariant set $E$ would contradict the results of  \cite{Barreira1996,GatPe97}.
\begin{cor}\label{cor:dynamical}
There exists a Lipschitz map $f \colon \R \to \R$, and $c>1$, $\delta > 0$, and $E \subset U \subset \R$ where $E$ is non-empty and compact and $U$ is bounded and open, such that the following three conditions hold:
\begin{itemize}
\item[(i)] for all $x \in U$ and $y \in (x-\delta,x+\delta)$ we have $|f(x)-f(y)| \geq c|x-y|$;
\item[(ii)] $E = f(E) = f^{-1}(E) \cap U$;
\item[(iii)] $\ldim E<\dim_{\mathrm H} E < \lbdim E < \ubdim E < \adim E$, where $\ldim$ and $\adim$ denote the lower and Assouad dimensions, respectively.
\end{itemize}
\end{cor}

Moreover, as a consequence of Theorem~\ref{thm:notexact}, there exist invariant measures for expanding Lipschitz dynamical systems which are not exact dimensional.
In a  related result, Pesin and Weiss \cite[Theorem~10]{PesinWeiss} constructed a H\"older continuous Axiom $A^{\#}$ homeomorphism of a compact subset of the plane, with positive topological entropy, for which the unique measure of maximal entropy is ergodic and has different upper and lower local dimension almost everywhere.
\begin{cor}\label{cor:dynamicalmeasure}
There exists a Lipschitz map $f \colon \R \to \R$, and $c>1$, $\delta > 0$, and $E \subset U \subset \R$ where $E$ is non-empty and compact and $U$ is bounded and open, and a Borel probability measure $\mu$ with $\supp(\mu) = E$, such that the following four conditions hold:
\begin{itemize}
\item[(i)] for all $x \in U$ and $y \in (x-\delta,x+\delta)$ we have $|f(x)-f(y)| \geq c|x-y|$;
\item[(ii)] $E = f(E) = f^{-1}(E) \cap U$;
\item[(iii)] $\mu$ is $f$-invariant and ergodic;
\item[(iv)] for all $x \in \supp(\mu)$ we have $\underline{\dim}(\mu,x) = \lbdim E < \ubdim E = \overline{\dim}(\mu,x)$, so $\mu$ is not exact dimensional.
\end{itemize}
\end{cor}

We finish this section by also recasting Theorem~\ref{thm:emptyinterior} in dynamical language. We will omit the proof of this statement as it is a repetition of the proofs of Corollaries \ref{cor:dynamical} and \ref{cor:dynamicalmeasure}.

\begin{cor}\label{cor:dynamicalinterior}
There exists a Lipschitz map $f \colon \R \to \R$, and $c>1$, $\delta > 0$, and $E \subset U \subset \R$ where $E$ is non-empty and compact and $U$ is bounded and open, such that the following three conditions hold:
\begin{itemize}
\item[(i)] for all $x \in U$ and $y \in (x-\delta,x+\delta)$ we have $|f(x)-f(y)| \geq c|x-y|$;
\item[(ii)] $E = f(E) = f^{-1}(E) \cap U$;
\item[(iii)] $E$ has empty interior and positive Lebesgue measure.
\end{itemize}
\end{cor}

\subsection{Open questions}
There are several natural questions for future study.
Recall that for our symmetric Cantor set construction, the contraction maps are non-differentiable at all points inside the Cantor set; see Proposition~\ref{prop-no-diff}.  Most likely, the contraction maps in the IFS constructed in Theorem~\ref{thm-5.1} are also non-differentiable at some or all points inside  the Cantor set. This raises the following question:
\begin{oques}
Is there an IFS on $\R^d$ consisting of differentiable contractions for which the box dimension of its attractor does not exist?
 \end{oques}
 Similarly, to the best of our knowledge, the versions of Questions~\ref{ques:hausdorffbox} and~\ref{ques:exactifs} for IFSs of differentiable contractions remain open.

 It is worth noting that for any given values $s,t\in (0,1]$ with $s\leq t$, there exists a bi-Lipschitz IFS whose attractor has lower and upper box dimensions  equal to $s$ and $t$, respectively; see Remark~\ref{rem-Cantor}. It would be interesting to investigate whether this result can be extended to all five notions of dimension. 

 Apart from these commonly used dimensions, there has been recent interest in studying parameterized families of dimensions which aim to interpolate between other notions of dimension~\cite{FraserInterpolation}. In particular, the intermediate dimensions lie between Hausdorff dimension and box dimension, and the Assouad spectrum lies between upper box dimension and Assouad dimension. It could be interesting to investigate the intermediate dimensions and Assouad spectrum of attractors of IFSs on $\R$, such as those  constructed in  Theorems~\ref{t1:dimE} and \ref{thm-5.1}, which exhibit distinct dimensions. It would also be interesting to study the Fourier dimension of these sets, which is likely to be positive, although such a study appears to be more challenging.

\subsection{Structure of the paper}

The paper is organized as follows. Section~\ref{S-2} recalls various notions of dimension, and summarizes general notation and conventions used in the paper. In Section~\ref{S-3}, we provide a sufficient condition for a Cantor set to be the attractor of a bi-Lipschitz IFS. In Section~\ref{S-4}, we construct a  bi-Lipschitz IFS whose attractor has distinct Hausdorff and box dimensions. In Section~\ref{S-5}, we restate Theorem~\ref{t:!bdim} as Theorem~\ref{thm-5.1} and provide a constructive proof.   In Section~\ref{Sec:simpler}, we present a simpler construction to generate symmetric Cantor sets, study the differentiability of the maps, and prove Theorems \ref{thm:notexact} and \ref{thm:emptyinterior}. In Section \ref{S-8}, we prove Corollaries \ref{cor:dynamical} and \ref{cor:dynamicalmeasure}.

\section{Notation and definitions of various dimensions}
\label{S-2}

First, let us recall the definitions of the following notions of dimension for subsets of $\R$: Hausdorff dimension, lower box dimension, upper box dimension, Assouad dimension, and lower dimension.

Let $E\subset \R$. The {\bf Hausdorff dimension} of  $E$ is defined by
\begin{equation*}\label{hausdorffdef}
\begin{aligned} \dim_\mathrm{H} E = \inf \{ s \geq 0 : &\mbox{ for every } \varepsilon >0 \mbox{ there exists a finite or countable cover } \\*
& \{U_1,U_2,\ldots\} \mbox{ of } E \mbox{ such that } \sum_i |U_i|^s \leq \varepsilon \},
\end{aligned}
\end{equation*}
where $|U|$ stands for the diameter of $U$.

 The  {\bf lower dimension} and {\bf upper box dimension} of  $E$ are defined respectively by
\[
\lbdim{E} = \liminf_{\delta\to0} \frac{\log N_{\delta}(E)}{-\log \delta}, \ \ \ \ubdim{E} = \limsup_{\delta\to0} \frac{\log N_{\delta}(E)}{-\log \delta},
\]
where $N_{\delta}(E)$ is the minimum number of intervals of diameter at most $\delta$ needed to cover $E$.
If $\underline{\dim}_{\mathrm B} E = \overline{\dim}_{\mathrm B} E$,  then we denote the common value by $\dim_{\mathrm B} E$ and call it the  {\bf box dimension} of $E$.

Moreover, the {\bf Assouad dimension} and {\bf lower dimension} of $E$ are defined respectively by
$$
\adim{E}=\limsup_{r\to 0}\sup_{R>0}\sup_{x\in E} \frac{\log  N_{Rr}(E\cap [x-R, x+R])}{\log(1/r)},
$$
and
$$
\ldim{E}=\liminf_{r\to 0}\inf_{0<R<1}\inf_{x\in E} \frac{\log N_{Rr}(E\cap [x-R, x+R])}{\log(1/r)}.
$$

It  always holds that $$\ldim E\le \hdim E\leq  \lbdim{E}\leq \ubdim{E}\leq \adim{E}\leq 1$$ for every $E\subset \R$; see e.g.~\cite{Fraser}.
     For further details on Hausdorff and box dimensions, see \cite{Falconer2014, Mattila1995},   and for  Assouad and lower dimensions, see \cite{Fraser}.

For the reader's convenience, we summarize in Table \ref{*table-1} the main notation and typographical conventions used in this paper.
\begin{table}
\centering
\caption{Main notation and conventions}
\label{*table-1}
\vspace{0.05 in}
\begin{footnotesize}
\begin{raggedright}
\begin{tabular}{p{1.5 in} p{4 in} }
\hline \rule{0pt}{3ex}
%& Effect of increasing \\
%Definition & substrate stiffness \\[3pt]
%\hline \rule{0pt}{3ex}
$B(x,r)$& Closed ball of radius $r$ centred at $x$\\
$\hdim$ & Hausdorff dimension (\S\ref{S-2})\\
$\ubdim,\;\lbdim$ & Upper and lower box dimensions (\S\ref{S-2})\\
$\adim,\;\ldim$ & Assouad dimension and lower dimension (\S\ref{S-2})\\
$\overline{\dim}(\mu,x),\;\underline{\dim}(\mu,x)$ & Upper and lower local dimensions of a measure $\mu$ at $x$ (\S\ref{S-1.3.1})\\
$\lfloor x\rfloor$ & Largest integer less than or equal to~$x$\\
$\Omega^\N$ & Full shift space over $\Omega\coloneqq \{0,1\}$\\
$\Omega^n$ & Collection of finite words over $\{0,1\}$ of length $n$\\
$\Omega^*$ & Collection of finite words over $\{0,1\}$\\
$[\omega]$ & Cylinder set in $\Omega^\N$ generated by a finite word $\omega$ (see \eqref{e:cyl})\\
$\omega^{-}$ & The prefix of a finite word $\omega$ obtained by removing its last symbol\\
 ${\mathsf L}_1(\omega)$ & Number of times the letter $1$ appears in $\omega$\\
  ${\mathsf N}_1(\omega)$ & Number of times the letter $1$ appears in $\omega_1\ldots\omega_{\lfloor\beta n\rfloor}$  for $\omega\in\Omega^n$\\
    ${\mathsf N}_2(\omega)$ & Number of times the letter $1$ appears in $\omega_{\lfloor\beta n\rfloor+1}\ldots\omega_{n}$  for $\omega\in\Omega^n$\\
 $\omega|n$ & $\omega|n \coloneqq \omega_1\ldots \omega_n$ for a word $\omega=(\omega_i)_{i=1}^k$ with $k\geq n$\\
 $I_\omega$, $\omega\in \Omega^*$ & Generating intervals of a Cantor set (\S\ref{S-3.1})\\
 $G_\omega$, $\omega\in \Omega^*$ & Gap intervals of a Cantor set (\S\ref{S-3.1})\\
 $|U|$ & Diameter of a set $U$\\
 $\pi\colon \Omega^\N\to E$ & Coding map associated with a Cantor set $E$ (\S\ref{S-3.2})\\
${\bf E}={\bf E}_{\beta,M}$ &  Cantor set corresponding to the length function $\ell(\cdot)$ given by \eqref{eq:length} \\
${\bf E}^*$ & Cantor set corresponding to the length function $\ell^*(\cdot)$ given by \eqref{e-ell*}  \\
$N_{\delta}(E)$ & Minimum number of intervals of length $\delta$ needed to cover $E$\\
$\#A$ & Cardinality of a finite set $A$\\
$H(\cdot)$ & Entropy function defined as in \eqref{e-entropy}\\
$D(\cdot)$ & A function defined as in \eqref{eq:DF}\\
$\Lambda(\rho)$ & Subset of $\Omega^*$ defined as in \eqref{e-lengthell}\\
$\Lambda^*(\rho),\;\Lambda_M(\rho),\; \Lambda_{M+1}(\rho)$ & Subsets of $\Omega^*$ defined as in \eqref{e-Lambda*M}\\
$\Lambda^{*,u}(\rho)$ & Subset of $\Omega^*$ defined as in \eqref{e-Lambda*u}\\
\hline
\end{tabular}
\end{raggedright}
\end{footnotesize}
\end{table}

\section{Cantor sets as attractors of bi-Lipschitz IFSs}
\label{S-3}

In this section, we  review the standard process for constructing a Cantor set. Furthermore, we provide a sufficient condition for a Cantor set to be the attractor of a bi-Lipschitz IFS; see Theorem~\ref{thm:bilip}.

We will adopt the standard multi-index notation for two-digit alphabets. Set $\Omega = \{0,1\}$. For $n\in \N$, let $\Omega^n$ be the $n$-fold Cartesian product of $\Omega$. Write $\Omega^0 = \{\varnothing\}$, where $\varnothing$ denotes  the empty word,  and $\Omega^{\ast} = \bigcup_{n=0}^{\infty} \Omega_n$. Let $\Omega^{\N}$ denote the countably infinite Cartesian product of $\Omega$. The concatenation of two words $\omega$ and $v$ is denoted by $\omega v$. If $\omega = \omega_1\ldots \omega_n\in\Omega^n$, then we denote by $\omega^- = \omega_1\ldots \omega_{n-1}$ the prefix of
$\omega$ obtained by removing its last symbol, and define
\begin{equation}\label{e:cyl}
[\omega] = \{ (v_i)_{i=1}^\infty \in \Omega^{\mathbb{N}}\colon v_i = \omega_i \mbox{ for } 1 \leq i \leq n \}.
\end{equation}
We call $[\omega]$ the {\bf cylinder set} in $\Omega^{\mathbb{N}}$ generated by $\omega$. As the length of $\omega$ is $n$, we also call $[\omega]$ a cylinder set of order $n$.

\subsection{Standard construction of  Cantor sets}
\label{S-3.1}
Recall that a subset of $\R$  is called a {\bf Cantor set} if it is compact, totally disconnected, and has no isolated points. We now review the standard process for constructing a Cantor set.  Let $I_{\varnothing} = [0,1]$.  Suppose that an interval $I_{\omega}$ has been constructed for some $\omega\in\Omega^{\ast}$. We decompose
\[
I_{\omega} = I_{\omega0}\cup G_{\omega} \cup I_{\omega1},
\]
where the union is disjoint,  $I_{\omega 0}$ and $I_{\omega1}$ are closed intervals, $G_{\omega}$ is an open interval,  the left endpoint of $I_{\omega0}$  coincides with the left endpoint of $I_{\omega}$ and  the right endpoint of $I_{\omega1}$  coincides with the right endpoint of $I_{\omega}$. Moreover, we assume that
\begin{equation}\label{eq-length-zero}
\lim_{n\to\infty} \max_{\omega\in\Omega^n}|I_\omega| = 0,
\end{equation}
 where we denote the length of an interval $I$ by $|I|$.
Define
\begin{equation}\label{eq_E}
E = \bigcap_{n=1}^{\infty} \bigcup_{\omega\in \Omega^n} I_{\omega}.
\end{equation}
Then $E$ is a Cantor set with $\{0,1\}\subset E\subset [0,1]$.
  Moreover,
\begin{equation}\label{eq-all}
[0,1] = E \cup \left(\bigcup_{\omega\in\Omega^{\ast}} G_{\omega}\right).
\end{equation}
We will call $\{I_{\omega}: \omega\in \Omega^{\ast}\}$ the {\bf generating intervals} of the Cantor set $E$ and $\{G_{\omega}: \omega\in\Omega^{\ast}\}$ the {\bf gaps} (or {\bf gap intervals}) of $E$. Notice that $G_{\omega}$, $\omega\in \Omega^{\ast}$,  are mutually disjoint. We remark that every Cantor subset of $[0,1]$ containing both $0$ and $1$ can be constructed in this way, and moreover, this Cantor is uniquely determined by the {\bf length function} $\omega\mapsto |I_\omega|$, $\Omega^*\to (0,1]$.

\subsection{A sufficient condition for  a Cantor set to be the attractor of a bi-Lipschitz IFS}
\label{S-3.2}
Now suppose that $E$ is a Cantor set defined by~\eqref{eq_E}.  We define the coding map $\pi \colon \Omega^{\N} \to E$ by
\begin{equation}
\label{e-pi}
\{\pi (\omega)\} = \bigcap_{n=1}^{\infty} I_{\omega_1\ldots \omega_n} \quad \mbox{for} \ \omega = (\omega_n)_{n=1}^\infty\in \Omega^{\N}.
\end{equation}
The map $\pi$ is one-to-one and surjective. For $x\in E$ and $n\in \N$, we write
\begin{equation}
\label{e-interval}
I_n(x) = I_{\omega_1\ldots \omega_n},
\end{equation}
where $\omega=(\omega_k)_{k=1}^\infty=\pi^{-1}(x)$ is the unique coding of $x$.

Next we introduce two maps $\phi_0$, $\phi_1$ from  $[0,1]$ to itself.
For $i=0,1$, we define $\phi_i:\; I_\varnothing=[0,1]\to I_i$ so that $\phi_i$ maps $G_{\omega}$ onto $G_{i\omega}$ for every $\omega\in \Omega^{\ast}$ as an affine map. More precisely, we define
\begin{equation}\label{eq:IFSmap}
\phi_i(x) = \left\{\begin{array}{ll}\frac{d-c}{b-a}(x-a)+c,  &  \text{if $x\in G_{\omega}$ for some $\omega\in\Omega^{\ast}$}, \\
\pi (i\omega),  &  \text{if $x = \pi (\omega)$ for some $\omega\in\Omega^{\mathbb N}$} ,
\end{array}\right.
\end{equation}
where we write $G_\omega=(a,b)$ and $G_{i\omega}= (c,d)$. By~\eqref{eq-all}, $\phi_i$ is well-defined and it maps $I_{\varnothing}$ onto $I_{i}$ for $i = 0,1$.  Moreover, it is readily checked that $\phi_0$ and~$\phi_1$ are strictly increasing and continuous. It is also straightforward to verify that
\begin{equation}
\label{e-Cantor}
 E=\phi_0(E)\cup\phi_1(E).
\end{equation}
Below we provide a necessary and sufficient condition for $\phi_i$, $i=0,1$, to be bi-Lipschitz contraction mappings on $[0,1]$.
 \begin{thm}\label{thm:bilip}
  Let $E$ be a Cantor subset of $[0,1]$ with generating intervals $\{I_\omega: \omega\in \Omega^*\}$ and gaps $\{G_\omega: \omega\in \Omega^*\}$. Let $\phi_i \colon [0,1]\to I_i$,  $i=0, 1$, be defined as above. Then for each $i\in \{0,1\}$, $\phi_i$ is a bi-Lipschitz contraction on $[0,1]$ if and only if
  \begin{equation}
  \label{e-thm2.1}
  0<\theta_* \coloneqq \inf_{\omega\in \Omega^*} \min\biggl\{ \frac{|I_{i\omega}|}{|I_\omega|}, \frac{|G_{i\omega}|}{|G_\omega|} \biggr\} \quad \mbox{and} \quad \theta^* \coloneqq \sup_{\omega\in \Omega^*} \max\biggl\{ \frac{|I_{i\omega}|}{|I_\omega|}, \frac{|G_{i\omega}|}{|G_\omega|} \biggr\} <1.
  \end{equation}
 \end{thm}

\begin{proof} We first prove the necessity part of the theorem. Assume that $\phi_0,\phi_1$ are bi-Lipschitz on $[0,1]$, and fix $i\in \{0,1\}$.  Notice that $\phi_i$ maps $G_{\omega} \eqqcolon (a,b)$ onto $G_{i\omega} = (\phi_i(a), \phi_i(b))$ for each $\omega\in \Omega^\ast$.  Applying the bi-Lipschitz assumption on $\phi_i$ to the points  $a$ and $b$,  we see that
\begin{equation}\label{eq:bilip_G}
 0<\inf_{\omega\in \Omega^{\ast}} \left\{\frac{|G_{i\omega}|}{|G_\omega|} \right\} \le \sup_{\omega\in \Omega^{\ast}} \left\{  \frac{|G_{i\omega}|}{|G_\omega|} \right\} <1.
 \end{equation}
 On the other hand, for each $\omega\in \Omega^{\ast}$, $I_\omega = [\pi (\omega0^{\infty}), \pi (\omega1^{\infty})]$, where $j^{\infty}= jjj \dots $ for $j\in\{0,1\}$ and $\pi$ is defined as in~\eqref{e-pi}. Therefore,
 \[
 I_{i\omega} =  [\pi (i\omega 0^{\infty}), \pi (i\omega1^{\infty})] = \phi_i(I_\omega).
 \]
 Applying the bi-Lipschitz assumption on $\phi_i$ to the points $\pi (\omega0^{\infty})$ and $\pi(\omega1^{\infty})$, we see that~\eqref{eq:bilip_G} also holds with $G_\omega$ replaced by $I_\omega$. Hence,  $0<\theta_*\le \theta^*<1$, where $\theta_*$ and $\theta^*$ are defined as in \eqref{e-thm2.1}.

 In what follows, we prove the sufficiency part of the theorem. To this end, assume that $0<\theta_*\le \theta^*<1$. Our goal is to prove that
\begin{equation}\label{eq:claim}
   \theta_*\le \frac{\phi_i(y)-\phi_i(x)}{y-x} \le  \theta^* \quad \mbox{ for all distinct } x,y\in[0,1].
  \end{equation}
  To justify~\eqref{eq:claim}, we let $0\le x<y\le 1$ and  $J=(x,y)$. First, notice that if there exists $\omega\in\Omega^{\ast}$ such that  $x,y\in \overline{G_\omega}$, then by the linearity of $\phi_i$ on $G_{\omega}$ and our assumption that $0<\theta_*\le \theta^*<1$, it is clear that~\eqref{eq:claim} holds. Moreover, this statement is evidently true for $x = 0$ and $y = 1$.

We now assume that $x,y$ are not in the closure of some common gap and $0 \leq x < y \leq 1$, and moreover either $x \neq 0$ or $y \neq 1$. Then $J\cap E\neq \varnothing$. Hence, $J$ contains at least one generating interval that is not $I_{\varnothing}$ since $J$ is open. We now let
\[
{\mathcal U} = \{ \omega \in\Omega^{\ast}\colon I_\omega\subset J, \ \mbox{and} \  I_{\omega^-}\not\subset J\}
\]
(recall that $\omega^{-}$ denotes the word obtained by removing the last letter of $\omega$) and
\[
J' = J\setminus \left(\bigcup_{\omega \in\mathcal{U}} I_{\omega}\right).
\]
The following lemma records some important properties of the sets $\mathcal{U}$ and $J'$.

\begin{lem}\label{lemma_structure}
\begin{itemize}
\item[(i)] $I_{\omega}\cap I_{v} = \varnothing$ for all $\omega,v\in{\mathcal U}$ with $\omega \ne v$.
\item[(ii)]  $E\cap J'= \varnothing$.
\item[(iii)] Suppose that $x,y\in E$. Then for every $\omega \in \Omega^{\ast}$, either $G_{\omega} \subset J'$ or $G_{\omega}\cap J' = \varnothing$.
\item[(iv)] Suppose that $x,y\in E$. Then there exists ${\mathcal V}\subset \Omega^{\ast}$ (${\mathcal V}$ may be empty) such that
\[
J  =\left(\bigcup_{\omega \in {\mathcal U}} I_{\omega}\right) \cup \left(\bigcup_{\omega \in {\mathcal V}} G_{\omega} \right).
\]
\end{itemize}
\end{lem}

\begin{proof}
We first prove (i).  By the construction of generating intervals,  the collection of all generating intervals $\{I_{\omega}:\omega \in\Omega^{\ast}\}$ of $E$ has the following net property: for all $\omega ,\omega'\in \Omega^{\ast}$, either
\begin{equation}\label{eq_net}
I_{\omega'}\subset I_{\omega}, \ \mbox{or} \ I_{\omega}\subset I_{\omega'}, \ \mbox{or} \ I_{\omega'}\cap I_{\omega} = \varnothing.
\end{equation}
Moreover, $I_{\omega'}\subset I_{\omega}$ occurs if and only if $\omega$ is a prefix of $\omega'$. Now suppose on the contrary that  $I_\omega\cap I_{v} \ne \varnothing$ for some distinct $\omega,v\in\mathcal{U}$.  By the aforementioned properties, we may assume, without loss of generality, that  $\omega\in \Omega^m$ and $v\in \Omega^n$, where  $m>n$. Then, again, the net property~\eqref{eq_net} implies that $I_\omega\subset I_v$ and $v$ is a prefix of $\omega$. Consequently we have $I_{\omega^-}\subset I_v\subset J$ since $v$ is also a prefix of $\omega^-$. This contradicts the fact that $\omega\in {\mathcal U}$. Hence, (i) holds.

To prove (ii), suppose on the contrary that $E\cap J'\neq  \varnothing$.  Take $z\in E\cap J'$ and set $\pi^{-1}(z)=(v_n)_{n=1}^\infty$, where $\pi$ is defined as in~\eqref{e-pi}.  Then
\[
\{z\} = \bigcap_{n=1}^{\infty} I_{v_1\ldots v_n}.
\]
Since $J$ is open, $I_{v_1\ldots v_{n}}\subset J$ when $n$ is large enough. Let $n_0$ be the smallest integer so that
$I_{v_1\ldots v_{n_0}}\subset J$. Then $v_1\ldots v_{n_0}\in\mathcal{U}$. This means that  $I_{v_1\ldots v_{n_0}}\cap J' = \varnothing$ and thus $z\not\in J'$. This results in a contradiction.

We now prove (iii). Let $\omega\in \Omega^*$.  Since the endpoints of $J$ are in $E$ and $G_\omega$ does not intersect $E$ since they are the gaps. We have either $G_\omega\subset J$ or $G_\omega\cap J = \varnothing$. We only need to consider the case when $G_\omega\subset J$ and show that either $G_\omega\subset J'$ or $G_\omega\cap J'= \varnothing$. Suppose the conclusion is false, i.e.,  $G_\omega\not\subset J'$ and $G_\omega\cap J'\neq \varnothing$. Since $G_\omega\not\subset J'$, it follows that
\[
G_\omega\cap \left(\bigcup_{v\in {\mathcal U}} I_{v}\right)\ne\varnothing.
\]
Hence, there exists $v\in \mathcal{U}$ such that $G_\omega\cap I_{v}\ne\varnothing$. However, since the two endpoints of $I_{v}$ are in $E$ and $G_\omega\cap E=\varnothing$, we have $G_\omega\subset I_{v}$, which implies that  $G_\omega\cap J' = \varnothing$, and  this leads to a contradiction.

Finally, we prove (iv). Since  $[0,1]$ is the disjoint union of $E$ and $\bigcup_{\omega\in \Omega^{*}}G_{\omega}$ by \eqref{eq-all},  statement (ii) implies that $J'\subset \bigcup_{\omega\in\Omega^{\ast}} G_\omega$. Set
 	\[
 	{\mathcal V}=\{\omega\in \Omega^*: G_{\omega} \subset J'\}.
 	\]
 Now using (ii) and (iii), our assumption that $x,y\in E$, and the fact that $J'\subset \bigcup_{\omega\in\Omega^{\ast}} G_{\omega}$, we must have $J'=\bigcup_{\omega \in {\mathcal V}} G_\omega$. Hence
 	\[
 	J=\left(\bigcup_{\omega \in{\mathcal U}} I_{\omega}\right) \cup J'=\left(\bigcup_{\omega \in{\mathcal U}} I_{\omega}\right) \cup \left(\bigcup_{\omega \in {\mathcal V}} G_{\omega} \right).
 	\]
 	This completes the proof of the lemma.
\end{proof}

We now return back to the proof of Theorem~\ref{thm:bilip}.  Let
	\[ x'=\min\bigl\{ t\in E\colon t\ge x \bigr\} \quad\text{and}\quad y'=\max \bigl\{ t\in E\colon t\le y \bigr\}. \]
	Then $(x,x')\cap E=\varnothing$, $(y',y)\cap E=\varnothing$ and $x',y'\in E$.  By Lemma~\ref{lemma_structure}(iv),
\[
(x',y') =  \left(\bigcup_{\omega \in {\mathcal U}} I_{\omega}\right) \cup \left(\bigcup_{\omega \in {\mathcal V}} G_{\omega} \right),
\]
where  the union is disjoint by Lemma~\ref{lemma_structure}(i). Therefore,
\[
\phi_i(x',y')= \left(\bigcup_{\omega \in {\mathcal U}} I_{i\omega} \right)\cup \left(\bigcup_{\omega \in {\mathcal V}} G_{i\omega}\right).
\]
Decomposing the interval $(x,y)$ as the disjoint union
$$(x,y) = (x,x'] \cup (x',y')\cup [y',y)$$  and using the fact that each $\phi_i$ is increasing, we have
\[
\phi_i(y)-\phi_i(x) =\phi_i(y)-\phi_i(y')+ \sum_{\omega \in{\mathcal U}} |I_{i\omega}| + \sum_{\omega \in{\mathcal V}}|G_{i\omega}|+ \phi_i(x')-\phi_i(x)
\]
We remark that either $x=x'$ or $(x,x')$ is contained in a gap. Similarly, $y=y'$ or $(y,y')$ is contained in a gap. Using these observations, together with the linearity of $\phi_i$ on the gap intervals and the definition of $\theta_*$ (see \eqref{e-thm2.1}),  we have
\[
 \phi_i(y)-\phi_i(x) \geq \theta_* \left( (y-y')+ \sum_{\omega \in\mathcal{U}} |I_{\omega}| + \sum_{\omega \in \mathcal{V}}|G_{\omega}| + (x'-x)\right) = \theta_*(y-x).
\]
A similar argument also works for the upper bound of~\eqref{eq:claim}. This completes the proof of Theorem~\ref{thm:bilip}.
\end{proof}

\section{Bi-Lipschitz IFSs whose attractors have distinct Hausdorff and box dimensions}

\label{S-4}
In this section, we construct a concrete class of Cantor sets, demonstrate that they are the attractors of bi-Lipschitz IFSs, and show that they have distinct Hausdorff and box dimensions.

Throughout the rest of the paper, $\lfloor x \rfloor$ denotes the largest integer less than or equal to~$x$.

Let $a_0=1$, $a_1=2$, $\beta\in (0,1)$, and let $M\geq 100$. Define a length function $\ell\colon\Omega^*\to(0,1]$ by
\begin{equation}\label{eq:length}
	\ell(\varnothing) = 1 \quad\text{and}\quad \ell(\omega_1\ldots \omega_n) = \frac{a_{\omega_1}\dotsm a_{\omega_{\lfloor\beta n\rfloor}}}{(a_{\omega_1}\dotsm a_{\omega_n})^\beta}\cdot M^{-n}.
\end{equation}

 Here we adopt the convention that when $m=0$, the empty product  $a_{\omega_1}\dotsm a_{\omega_m}$ is equal to $1$. Let ${\bf E} = {\bf E}_{\beta,M}$ be the Cantor set generated by the length function~$\ell$, i.e., the generating intervals $\{I_\omega\}_{\omega\in\Omega^*}$ of ${\bf E}$ satisfy $|I_\omega|=\ell(\omega)$. In the remaining part of this section, ${\bf E}$ with boldface always denotes the Cantor set obtained from this length function.
It is easy to see that $\max_{\omega\in \Omega^{n}}|I_\omega| \to 0$ as $n \to \infty$. This shows that ${\bf E}$ is a well-defined Cantor set.

 The definition of $\ell(\omega)$ is inspired by the dimension calculations of Bedford--McMullen carpets in~\cite{McMul84}.
 Indeed, for a certain self-affine measure $\eta$ supported on a Bedford--McMullen carpet, the $\eta$-measure of an approximate square has an expression similar to $\ell(\omega)$; see \cite[Lemma~3]{McMul84}.

The following proposition shows that ${\bf E}$ is the attractor of a bi-Lipschitz IFS.
\begin{prop}\label{p:BLIFS}
Let $\beta\in (0,1)$ and $M\geq 100$. 	Let ${\bf E}={\bf E}_{\beta,M}$ be the Cantor set defined as above. Then the collection $\Phi=\{\phi_0,\phi_1\}$, where each map $\phi_i$ is defined via the Cantor construction of ${\bf E}$ as in \eqref{eq:IFSmap}, forms a bi-Lipschitz IFS whose attractor is ${\bf E}$.
\end{prop}

The proof of Proposition~\ref{p:BLIFS} is an application of Theorem~\ref{thm:bilip} and will be given in Section~\ref{S-IFSattractor}.

Moreover, we are able to determine the precise values of the Hausdorff and  box dimensions of ${\bf E}$.  To state our result, we first introduce some notation. The entropy function $H\colon [0,1]\to \R$ is defined by
\begin{equation}
\label{e-entropy}
H(p) = -p\log p - (1-p) \log (1-p),
\end{equation}
with the convention that $0\log 0 = 0$. For $\lambda>\log 2$, define
\begin{equation}\label{eq:DF}
	D(\lambda) = \max_{(p,q)\in[0,1]^2} \frac{\beta H(p) + (1-\beta) H(q)}{\lambda + \beta(1-\beta)(q-p)\log 2}.
\end{equation}
 The maximum in \eqref{eq:DF} is attained because the function being maximized is continuous on the compact set  $[0,1]^2$.

 Now we are ready to present our result on the Hausdorff and box dimensions of ${\bf E}$.

\begin{thm}\label{t1:dimE}
	Let ${\bf E}={\bf E}_{\beta, M}$, where $\beta\in(0,1)$ and $M\ge 100$.  Then
	\begin{itemize}
		\item[(i)]  $\hdim {\bf E} = \frac{\log 2}{\log M}$;
		\item[(ii)]  $\bdim {\bf E} = D(\log M)$.
\item[(iii)] $\hdim {\bf E}<\bdim {\bf E}$.				
	\end{itemize}
\end{thm}

The proofs of parts (i), (ii), (iii) of  the above theorem will be given separately in Sections \ref{S-4.2}, \ref{S-4.3} and \ref{S-4.4}, respectively.

\subsection{The proof of Proposition \ref{p:BLIFS}}
\label{S-IFSattractor}

\begin{proof}[Proof of Proposition \ref{p:BLIFS}]

 We now prove that the Cantor set ${\bf E} = {\bf E}_{\beta,M}$, with length function given by~\eqref{eq:length}, is generated by a bi-Lipschitz IFS. To this end, we will show that the assumptions in Theorem~\ref{thm:bilip} are satisfied.

 By \eqref{eq:length},  a direct computation shows that for $\omega\in \Omega^n$ and $i\in \{0,1\}$, 
\begin{equation*} \label{eq-ratios-1}
\frac{|I_{i\omega}|}{|I_{\omega}|}  = \begin{cases}
		u_i M^{-1}, & \text{if $\lfloor\beta(n+1)\rfloor=\lfloor\beta n\rfloor=0$}, \\
		v_ia_{\omega_{\lfloor\beta n\rfloor}}^{-1}M^{-1}, & \text{if $\lfloor\beta(n+1)\rfloor=\lfloor\beta n\rfloor\ge1$}, \\
		v_i M^{-1}, & \text{if $\lfloor\beta(n+1)\rfloor=\lfloor\beta n\rfloor+1$},
	\end{cases}
	\end{equation*}
where $u_0=v_0=1$, $u_1 = 2^{-\beta}$ and $v_1 = 2^{1-\beta}$; and
\begin{equation*} \label{eq-ratios-2} \frac{|I_{\omega i}|}{|I_\omega|}  = \begin{cases}
		 s_i\cdot \frac{1}{M}, & \text{if $\lfloor\beta(n+1)\rfloor=\lfloor\beta n\rfloor$}, \\
		s_i\cdot a_{\omega_{\lfloor\beta n\rfloor+1}}\cdot \frac1{M}, & \text{if $\lfloor\beta(n+1)\rfloor=\lfloor\beta n\rfloor+1$},
	\end{cases}
		\end{equation*}
where $s_0 = 1$ and $s_1 = 2^{-\beta}$.  It follows that
\begin{equation} \label{eq-ratios-3}
  \frac{1}{2 M} \le  \frac{|I_{i\omega}|}{|I_w|} \le \frac{2}{M} \quad \mbox{ and }\quad \frac{1}{2 M} \le  \frac{|I_{\omega i}|}{|I_w|} \le \frac{2}{M}\end{equation}
 for all $\omega\in \Omega^{\ast}$  and $i\in \{0,1\}$. Consequently,
 \begin{equation*}
		\begin{aligned}
			\inf_{\omega\in\Omega^*}\min\biggl\{ \frac{|I_{0\omega}|}{|I_\omega|}, \frac{|I_{1\omega}|}{|I_\omega|} \biggr\} &\ge \frac{1}{2M}>0, \\
\sup_{\omega\in\Omega^*}\max\biggl\{ \frac{|I_{0\omega}|}{|I_\omega|}, \frac{|I_{1\omega}|}{|I_\omega|} \biggr\} &\le \frac{2}{M}<1.
		\end{aligned}
	\end{equation*}

It remains to show that an analogous result holds for the gap intervals. To this end, we note that $|G_\omega|=|I_\omega|-|I_{\omega0}|-|I_{\omega1}|$ and
	\[ \frac{|G_{i\omega}|}{|G_\omega|} = \frac{|I_{i\omega}|}{|I_\omega|}\cdot \frac{1-|I_{i\omega0}|/|I_{i\omega}|-|I_{i\omega1}|/|I_{i\omega}|}{1-|I_{\omega0}|/|I_\omega|-|I_{\omega1}|/|I_\omega|}. \]
	By~\eqref{eq-ratios-3}, we obtain
	\begin{equation*}
		\begin{aligned}
			\inf_{\omega\in\Omega^*} \min \biggl\{ \frac{|G_{0\omega}|}{|G_\omega|}, \frac{|G_{1\omega}|}{|G_\omega|} \biggr\} &\ge \frac{1}{2M} \cdot \frac{M-4}{M-1}>0, \\
			\sup_{\omega\in\Omega^*} \max \biggl\{ \frac{|G_{0\omega}|}{|G_\omega|}, \frac{|G_{1\omega}|}{|G_\omega|} \biggr\} &\le\frac{2}{M}\cdot \frac{M- 1}{M- 4}<1.
		\end{aligned}
	\end{equation*}
 Hence, the condition listed in \eqref{e-thm2.1} is verified.  This completes the proof.
\end{proof}

\subsection{Hausdorff dimension of ${\bf E}$}
\label{S-4.2}
 In this subsection, we prove Theorem~\ref{t1:dimE}(i). We will use the following version of Billingsley's lemma \cite[Lemma~1.4.1]{BisPe17} (from the remarks given after Lemma~1.4.1 in \cite{BisPe17}, we know that the lemma also holds when using the generating intervals instead of dyadic intervals).

\begin{lem}[Billingsley's Lemma] \label{Billingsley-lemma}
Let $\mu$ be a finite Borel measure on a Cantor set $E\subset [0,1]$ whose generating intervals are given by $\{I_{\omega}: \omega\in \Omega^{\ast}\}$. Suppose that
$$
\inf_{\omega\in\Omega^*}\min\biggl\{ \frac{|I_{\omega 0}|}{|I_\omega|}, \frac{|I_{\omega 1}|}{|I_\omega|} \biggr\} >0.
$$
Let   $A\subset E$ with $\mu (A)>0$. Suppose that
\begin{equation}
\label{e-Billingsley}
 \liminf_{n\to\infty} \frac{\log \mu(I_{n} (x))}{\log |I_{n}(x)|} \le\alpha\quad \mbox{for all }x\in A,
\end{equation}
where $I_n(x)$ is defined as in~\eqref{e-interval}. Then $\hdim A\le \alpha$. Moreover, if the inequality $\le$ in \eqref{e-Billingsley} is replaced by $\ge$, then $\hdim A \geq \alpha$.
\end{lem}

As well as Billingsley's lemma, we need the following result due to McMullen \cite{McMul84}.

\begin{lem}\cite[Lemma~4]{McMul84} \label{mcmullen-lemma} Let $a_0 = 1$,  $a_1 = 2$, and $\beta\in (0,1)$.  Then
\begin{equation}
\label{e-Mcmullen}
	\limsup_{n\to\infty}\biggl( \frac{(a_{\omega_1}\dotsm a_{\omega_n})^\beta}{a_{\omega_1}\dotsm a_{\omega_{\lfloor\beta n\rfloor}}} \biggr)^{1/n} \ge 1 \quad\mbox{for all }\;\omega= (\omega_n)_{n=1}^{\infty}\in \Omega^{\N}.
	\end{equation}
Moreover, for every Bernoulli product measure $\nu$ on $\Omega^{\N}$,
\[
\lim_{n\to\infty}\biggl( \frac{(a_{\omega_1}\dotsm a_{\omega_n})^\beta}{a_{\omega_1}\dotsm a_{\omega_{\lfloor\beta n\rfloor}}} \biggr)^{1/n} = 1
\]
for $\nu$-a.e.~$\omega\in \Omega^{\N}$.
\end{lem}

\begin{proof}[Proof of Theorem~\ref{t1:dimE}(i)] Let $\nu$ be the Bernoulli product measure on $\Omega^\N$ that assigns measure $2^{-n}$ to every cylinder set of order $n$. Let $\mu=\pi_*\nu$ be the push-forward of $\nu$ by the coding map $\pi$. Then $\mu$ is supported on ${\bf E}$ and
\[
\mu(I_\omega)= 2^{-n} \ \ \mbox{ for all } \omega\in \Omega^{n},\;  n\in \N.
\]
For $x\in {\bf E}$, let $(\omega_n)_{n=1}^\infty=\pi^{-1}(x)$ be the coding of $x$. Recalling the definition of $\ell(\cdot)$ in~\eqref{eq:length}, we see that
	\begin{equation}\label{eq:ALL}
\begin{split}
	\limsup_{n\to\infty} \frac{\log |I_n(x)|}{-n} &= \limsup_{n\to\infty} \frac{\log \ell(\omega_1\ldots \omega_n)}{-n} \\
&=\limsup_{n\to\infty}~\log \biggl( \frac{(a_{\omega_1}\dotsm a_{\omega_n})^\beta}{a_{\omega_1}\dotsm a_{\omega_{\lfloor\beta n\rfloor}}} \biggr)^{1/n} + \log M\\
&\geq \log M \qquad\qquad\mbox{(by \eqref{e-Mcmullen})}.
\end{split}	
\end{equation}
 Hence
	\begin{equation}
		\label{eq:Billingsley uniform upper bound}
	 \liminf_{n\to\infty} \frac{\log \mu(I_{n} (x))}{\log |I_{n}(x)|} \le \frac{\log 2}{\log M} \quad \mbox{ for all }x\in {\bf E}.
	\end{equation}
By Lemma~\ref{Billingsley-lemma}, $\hdim {\bf E} \le \log 2/\log M$. On the other hand, the second part of  Lemma~\ref{mcmullen-lemma} implies that the third inequality in \eqref{eq:ALL} is an equality for $\mu$-a.e.~$x$. Hence
\begin{equation*}
	 \liminf_{n\to\infty} \frac{\log \mu(I_{n} (x))}{\log |I_{n}(x)|} = \frac{\log 2}{\log M} \quad \mbox{ for $\mu$-a.e.~}x\in {\bf E}.
	\end{equation*}
By Lemma~\ref{Billingsley-lemma}, there exists a set $A\subset {\bf E}$ with $\mu(A) = \mu({\bf E})$ attaining the  Hausdorff dimension $\log 2/\log M$. Therefore $\hdim {\bf E} = \log 2/\log M$. This completes the proof of Theorem~\ref{t1:dimE}(i).
\end{proof}

\subsection{Box dimension of ${\bf E}$}
\label{S-4.3}
In this subsection, we prove Theorem~\ref{t1:dimE}(ii).

We begin with the following known result, which is a direct consequence of Stirling's formula. See e.g.~\cite[Lemma 2]{FengWang2005} for a proof.

\begin{lem}
\label{lem-4.3}
There exists $c>0$ such that for any integers $n,k$ with $n\geq 1$ and $0\leq k\leq n$,
\begin{equation}
\label{e-4.5}
n^{-c}\exp(nH(k/n))\leq \binom{n}{k}\leq n^c\exp(nH(k/n)),
\end{equation}
where $H(\cdot)$ is defined as in \eqref{e-entropy}.
\end{lem}

Recall that the length function $\ell(\cdot)$ is defined as in \eqref{eq:length}. For $0<\rho<1$, define
  \begin{equation}
  \label{e-lengthell}
\Lambda(\rho) = \left\{\omega\in\Omega^{\ast}\colon \ell(\omega)\le \rho <\ell(\omega^{-})\right\}.
\end{equation}
The following lemma is also needed in our proof of Theorem~\ref{t1:dimE}(ii).

\begin{lem}
\label{lem-4.7}
\begin{itemize}
\item[(i)] The upper and lower box dimensions of ${\bf E}$ satisfy the following:
\begin{equation}\label{eq-upper-box-1}
\ubdim {\bf E} =  \limsup_{\rho\to 0} \frac{\log(\# \Lambda(\rho))}{\log (1/\rho)},\qquad \lbdim {\bf E} =  \liminf_{\rho\to 0} \frac{\log(\# \Lambda(\rho))}{\log (1/\rho)},
\end{equation}
where $\#A$ stands for the cardinality of $A$.
\item[(ii)] Assume that $0<\rho<\frac{1}{2M}$. If $\Omega^n\cap \Lambda(\rho)\neq\varnothing$ for some $n\in \N$, then
\begin{equation}
\label{e-Nrho}
\frac{\log(1/\rho)}{\log (2 M)}\leq n\leq \frac{2\log(1/\rho)}{\log (M/2)}.
\end{equation}
\end{itemize}
\end{lem}
\begin{proof}
Let $0<\rho<1$. By the definition $\Lambda(\rho)$, every element $x=(x_n)_{n=1}^\infty\in \Omega^\N$ (which can be viewed as an infinite word) has a unique initial prefix in $\Lambda(\rho)$. It follows that
$$
\Omega^\N=\bigcup_{\omega\in\Lambda(\rho)}[\omega],
$$
where the union is disjoint. As the coding map $\pi\colon \Omega^\N\to {\bf E}$ satisfies  $\pi (x)\in I_{x_1\ldots x_n}$ for each $x\in \Omega^\N$ and $n\geq 1$, it follows that  the sub-collection $\{I_\omega\colon \omega\in\Lambda(\rho)\}$ of generating intervals forms a disjoint cover of ${\bf E}$. Meanwhile by the second inequality in \eqref{eq-ratios-3} and the definition of $\Lambda(\rho)$, for each $\omega\in \Lambda(\rho)$,
\begin{equation}
\label{e-ellomega-length}
\frac{\rho}{2 M}<\frac{1}{2 M} \ell(\omega^-)\leq\ell(\omega)\leq \rho.
\end{equation}
This implies, in particular, that  the diameter of each element in the sub-collection $\{I_\omega\colon \omega\in\Lambda(\rho)\}$ does not exceed $\rho$, and each interval of length $\frac{\rho}{2 M}$ intersects at most two elements in this sub-collection.
Therefore,
 \begin{equation}
 \label{e-NE}
 N_{\rho}({\bf E})\le \#\Lambda (\rho) \quad\mbox{ and }\quad N_{\frac{\rho}{2 M}}({\bf E})\geq \frac{1}{2}\cdot\big(\#\Lambda (\rho)\big),
 \end{equation}
 where  $N_{\rho}({\bf E})$ is the minimum number of intervals of diameter at most $\rho$ needed to cover ${\bf E}$. It follows that
 $$\frac{1}{2}\cdot\big(\#\Lambda (2 M\rho)\big)\leq N_{\rho}({\bf E})\leq \#\Lambda (\rho)$$
 for each $0<\rho<(2 M)^{-1}$, which implies \eqref{eq-upper-box-1}. This proves (i).

To prove (ii), by  \eqref{eq-ratios-3},
$$
(2 M)^{-n}\leq \ell(\omega)\leq (M/2)^{-n} \quad \mbox{ for all }\omega\in \Omega^n.
$$
Hence, if $\omega\in \Omega^n\cap \Lambda(\rho)$, then by \eqref{e-ellomega-length},
$$
(2 M)^{-n}\leq \ell(\omega)\leq \rho\quad \mbox{ and }\quad (M/2)^{-n}\geq \ell(\omega)\geq \frac{\rho}{2 M},
$$
which implies
\begin{equation*}
\frac{\log(1/\rho)}{\log (2 M)}\leq n\leq \frac{\log (2 M \rho^{-1})}{\log (M/2)}\leq \frac{2\log(1/\rho)}{\log (M/2)},
\end{equation*}
where in the last inequality we have used the assumption $\rho\leq (2 M)^{-1}$. This completes the proof of (ii).
\end{proof}

Now we are ready to prove Theorem~\ref{t1:dimE}(ii).

\begin{proof}[Proof of Theorem \ref{t1:dimE}(ii)]
Write for brevity $D=D(\log M)$.   Below we divide our proof into two steps.

{\it Step 1. $\ubdim {\bf E}\leq D$}.  For $n\in \N$ and $\omega = \omega_1\ldots\omega_n\in \Omega^n$, write
\begin{equation}
\label{e-4split}
\begin{split}
{\mathsf N}_1(\omega) &=  \#\{1\le j\le \lfloor\beta n\rfloor\colon \omega_j = 1\},\\
 {\mathsf N}_2(\omega)&= \#\{\lfloor\beta n\rfloor+1\le j\le n\colon \omega_j = 1\}.
\end{split}
\end{equation}
Let $0<\rho<1$. We claim that if $\omega\in \Omega^n\cap \Lambda(\rho)$ for some $n\in \N$, then
\begin{equation}
\label{e-4.11}
\left( 2^{\beta(1-\beta)\left(\frac{{\mathsf N}_2(\omega)}{n-\lfloor\beta n\rfloor}-\frac{{\mathsf N}_1(\omega)}{\lfloor\beta n\rfloor}\right)} M\right)^n\leq \frac{2 M}{\rho}.
\end{equation}
To see this, let $\omega\in \Omega^n\cap \Lambda(\rho)$.  By \eqref{eq:length},
\begin{equation}\label{eq:length2}
\ell(\omega) = \frac{a_{\omega_1}\dotsm a_{\omega_{\lfloor\beta n\rfloor}}}{(a_{\omega_1}\dotsm a_{\omega_n})^\beta}\cdot  M^{-n} = 2^{(1-\beta){\mathsf N}_1(\omega)-\beta {\mathsf N}_2(\omega)}  M^{-n}.
\end{equation}
Writing $p=\frac{{\mathsf N}_1(\omega)}{\lfloor\beta n\rfloor}$ and $q=\frac{{\mathsf N}_2(\omega)}{n-\lfloor\beta n\rfloor}$,  we see that $p,q\in [0,1]$ and
\begin{align*}
\frac{2 M}{\rho}&\geq \ell(\omega)^{-1}\qquad\qquad \qquad\qquad\qquad\qquad\qquad\qquad\qquad\mbox{(by \eqref{e-ellomega-length})}\\
&=2^{\beta {\mathsf N}_2(\omega)-(1-\beta){\mathsf N}_1(\omega)}  M^{n}\qquad\qquad\qquad\qquad \qquad\qquad\mbox{(by \eqref{eq:length2})}\\
&=2^{\beta (n-\lfloor \beta n\rfloor)q-(1-\beta)\lfloor\beta n\rfloor p}  M^{n}\\
&=2^{n\beta (1-\beta)(q-p)} 2^{(\beta q+(1-\beta)p)(n\beta-\lfloor n \beta\rfloor)} M^{n}\\
&\geq 2^{n\beta (1-\beta)(q-p)} M^{n}.
\end{align*}
This proves \eqref{e-4.11}.

Now write
$$
{\mathcal N}_{\rho}=\left\{n\in \N\colon \Omega^n\cap \Lambda(\rho)\neq \varnothing\right\},
$$
and
$${\mathcal U}_n =  \{0, 1,\dotsc, \lfloor\beta n\rfloor\}\times \{0,\dotsc, n-\lfloor\beta n\rfloor\},\quad n\in \N.$$ Clearly, we have
\begin{equation}\label{eq-Lambda-rho}
\# \Lambda(\rho) = \sum_{n\in {\mathcal N}_{\rho}}\sum_{(N_1,N_2)\in {\mathcal U}_n} \#\left\{\omega\in\Omega^n\cap \Lambda(\rho)\colon {\mathsf N}_1(\omega) = N_1, {\mathsf N}_2(\omega) = N_2\right\}.
\end{equation}
  Observe that for $n\in {\mathcal N}_{\rho}$ and $(N_1,N_2)\in {\mathcal U}_n$,
\begin{align*}
\# &\left\{\omega\in \Omega^n \cap \Lambda(\rho)\colon {\mathsf N}_1(\omega) = N_1, {\mathsf N}_2(\omega) = N_2\right\} \\
&\le   ~\binom{\lfloor\beta n\rfloor}{N_1}\cdot \binom{n-\lfloor\beta n\rfloor}{N_2}\\
&\leq n^{2c}  \exp\left( \lfloor\beta n\rfloor H\left(\frac{N_1}{\lfloor\beta n\rfloor}\right)+ (n-\lfloor n\beta\rfloor) H\left(\frac{N_2}{n-\lfloor\beta n\rfloor}\right)\right) \qquad\mbox{(by \eqref{e-4.5})}\\
&\leq n^{2c}\cdot 2\cdot\exp\left(n\left( \beta  H\left(\frac{N_1}{\lfloor \beta n\rfloor}\right)+ (1-\beta) H\left(\frac{N_2}{n-\lfloor\beta n\rfloor}\right)\right)\right)\\
&\leq n^{2c}\cdot 2\cdot \exp\left(Dn\left( \beta (1-\beta) \left(\frac{N_2}{n-\lfloor\beta n\rfloor}-\frac{N_1}{\lfloor\beta n\rfloor}\right)\log 2 +\log M\right)\right)\\
&=n^{2c}\cdot 2\cdot  \left(2^{\beta (1-\beta)\left(\frac{N_2}{n-\lfloor\beta n\rfloor}-\frac{N_1}{\lfloor\beta n\rfloor}\right)} M\right)^{Dn}\\
&\leq n^{2c}\cdot 2\cdot \left(\frac{2 M}{\rho}\right)^{D}\qquad\qquad\qquad\qquad\qquad\qquad\qquad\qquad  \mbox{(by \eqref{e-4.11})},
\end{align*}
where in the fourth inequality, we have used the definition of $D=D(\log M)$ (see \eqref{eq:DF}).
For $0<\rho\leq (2 M)^{-1}$, plugging the above inequality into \eqref{eq-Lambda-rho} and using the facts that $\#{\mathcal U}_n\leq n^2$ and $$n\leq \frac{2\log(1/\rho)}{\log (M/2)}$$ for each $n\in \mathcal N_\rho$  (see \eqref{e-Nrho}),
we obtain
$$
\# \Lambda(\rho)\leq \left(\frac{2\log(1/\rho)}{\log (M/2)}\right)^{3+2c} \cdot 2 \cdot \left(\frac{2 M}{\rho}\right)^{D}.
$$
Combining this with \eqref{eq-upper-box-1} yields that $\ubdim {\bf E}\leq D$.
\medskip

{\it Step 2:  $\lbdim {\bf E}\geq D$}.  Let $p,q\in (0,1)$ and
write $$\gamma=\gamma_{p,q}=\log M+\beta(1-\beta)(q-p)\log 2.$$
For $n\in \N$, define
$$
G_n=\left\{\omega \in \Omega^n\colon {\mathsf N}_1(\omega)=\lfloor p\beta n\rfloor \mbox{ and } {\mathsf N}_2(\omega) =\lfloor q(1-\beta)n\rfloor\right\},
$$
where ${\mathsf N}_1(\omega),{\mathsf N}_2(\omega)$ are defined as in \eqref{e-4split}. By~\eqref{eq:length2},  for each $\omega\in G_n$,
$$
 \delta_n: = \ell(\omega) = e^{-n \left(\log M-(1-\beta)\frac{\lfloor p\beta n\rfloor }{n}\log 2+\beta \frac{\lfloor q(1-\beta)n\rfloor}{n}\log 2\right)}   = e^{-n(\gamma+ \epsilon_n)}
$$
where
\[
\epsilon_n = \beta(1-\beta) \log 2 \cdot  \left(\left(\frac{\lfloor q(1-\beta)n\rfloor }{(1-\beta)n}-q\right)-\left(\frac{\lfloor p\beta n\rfloor}{\beta n}-p\right)\right).
\]
  As $\left|\frac{\lfloor p\beta n\rfloor}{\beta n}-p\right|\le\frac{1}{\beta n}$ and $\left|\frac{\lfloor q(1-\beta)n\rfloor}{(1-\beta)n}-q\right|\le\frac{1}{(1-\beta) n}$,
  \[
  |\epsilon_n|\le \frac{(1-\beta)\log 2}{n}+\frac{\beta \log 2}{n}= \frac{\log 2}{n} .
  \]
Since $\delta_n=e^{-n(\gamma+ \epsilon_n)}$, this shows that
\[
\frac12e^{-n\gamma}\leq \delta_n\leq 2e^{-n\gamma}.
\]
Note that the intervals in the collection $\{I_\omega \colon \omega\in G_n\}$ are disjoint, all have the same length $\delta_n$, and each contains points in ${\bf E}$.
So each interval of length $\delta_n$ intersects at most two elements in this collection. It follows that
\[
N_{\delta_n}({\bf E}) \ge \frac{1}{2}(\#G_n)  = \frac{1}{2}\binom{\lfloor\beta n\rfloor}{\lfloor p\beta n\rfloor}\cdot \binom{n-\lfloor\beta n\rfloor}{\lfloor q(1-\beta)n\rfloor}.
  \]
By \eqref{e-4.5}, for all  sufficiently large $n$,
\begin{align*}
\frac{1}{2}& \binom{\lfloor\beta n\rfloor}{\lfloor p\beta n\rfloor}\cdot \binom{n-\lfloor\beta n\rfloor}{\lfloor q(1-\beta)n\rfloor}\\ &\ge \frac{1}{2}n^{-2c} \cdot
\exp\left(\lfloor \beta n\rfloor H\left(\frac{\lfloor p\beta n\rfloor}{\lfloor \beta n\rfloor}\right)+(n-\lfloor\beta n\rfloor)H\left(\frac{\lfloor q(1-\beta)n\rfloor}{n-\lfloor \beta n\rfloor}\right)\right)\\
&\geq n^{-3c} \exp(n(\beta H(p)+(1-\beta)H(q))),
\end{align*}
where $c$ is the constant in \eqref{e-4.5}, and in the second inequality we have used the estimates
$$
\left|H\left(\frac{\lfloor p\beta n\rfloor}{\lfloor \beta n\rfloor}\right)-H(p)\right|= O\left(\frac{1}{n}\right), \qquad \left|H\left(\frac{\lfloor q(1-\beta)n\rfloor}{n-\lfloor\beta n\rfloor}\right)-H(q)\right|= O\left(\frac{1}{n}\right),
$$
which follow from the finiteness of the derivatives of $H$ at $p$ and $q$. Here we use the standard big $O$ notation.
Hence, for all sufficiently small $\delta>0$, by choosing  $n$ such that $\delta\in(\frac12e^{-(n+1)\gamma},\;\frac12e^{-n\gamma})$, we obtain
\[
\frac{\log N_{\delta}({\bf E})}{-\log \delta}\ge \frac{\log N_{\delta_n}({\bf E})}{-\log (\frac12e^{-(n+1)\gamma})} \ge \frac{\log (\frac{1}{2}\#G_n)}{(n+1)\gamma+\log 2}. \]
Using our estimate for $\#G_n$, it follows that
\[\frac{\log N_{\delta}({\bf E})}{-\log \delta}\ge \frac{-3c\log n+n(\beta H(p)+(1-\beta)H(q))}{(n+1)\gamma+\log 2}.
\]
Taking the lower limit gives
\[
\lbdim {\bf E}\ge \frac{\beta H(p)+ (1-\beta)H(q)}{\gamma}  = \frac{\beta H(p)+(1-\beta)H(q)}{\log M+\beta(1-\beta) (q-p)\log 2}.
\]
Hence, $\lbdim {\bf E}\ge\Phi (p,q)$ for all $(p,q)\in (0,1)^2$, where
\[
\Phi(p,q) = \frac{\beta H(p)+(1-\beta)H(q)}{\log M+\beta(1-\beta) (q-p)\log 2}.
\]
Noting that $\Phi$ is a continuous function on $[0,1]^2$ and $D = \max_{[0,1]^2}\Phi(p,q)$,  we obtain that  $\lbdim {\bf E}\ge D$. This completes our proof.
\end{proof}

\begin{rem}
\label{rem-4.4}
By Theorem~\ref{t1:dimE}(ii) and \eqref{eq-upper-box-1},
$$
\bdim {\bf E}=\lim_{\rho\to 0} \frac{\log (\#\Lambda(\rho))}{\log (1/\rho)}.
$$
\end{rem}

\subsection{The proof of Theorem \ref{t1:dimE}(iii)}
\label{S-4.4}
We begin with a simple lemma.
\begin{lem}\label{l:DF1}
Let $D(\cdot)$ be defined as in \eqref{eq:DF}. Then for all  $\lambda > \log 2$,
\[
D(\lambda) > \frac{\log 2 }{ \lambda}.
\]
\end{lem}

\begin{proof}
 On the right-hand side of \eqref{eq:DF}, by taking $p =\frac12$ and  $q = \frac12-\varepsilon$ for some small $\varepsilon>0$,  we have
$$D(\lambda) \ge \frac{\beta H(\frac12)+(1-\beta)H(\frac12-\epsilon)}{\lambda - \beta(1-\beta)\varepsilon\log 2}.
$$
Note that
\[
H\left(\frac12-\varepsilon\right) = -\left(\frac12+\varepsilon\right)\log \left(\frac12+\varepsilon\right) -\left(\frac12-\varepsilon\right)\log \left(\frac12-\varepsilon\right).
\]
Expanding $\log \left(\frac12+\varepsilon\right)  = -\log2+ 2\varepsilon -2\varepsilon^2+O(\varepsilon^3)$ as a Taylor series, we see that
$H(\frac12-\varepsilon) = \log 2- 2\varepsilon^2+ O(\varepsilon^3)$. Hence,
\begin{equation*}
D(\lambda) \geq  \frac{\log 2-2(1-\beta)\varepsilon^2+O(\varepsilon^3)}{\lambda - \beta(1-\beta)\varepsilon\log 2}.
\end{equation*}
Since the right-hand side of the above inequality exceeds $\frac{\log 2}{\lambda}$ when $\varepsilon$ is sufficiently small,  we have $D(\lambda)>\frac{\log 2}{\lambda}$.
\end{proof}

\begin{proof}[Proof of Theorem \ref{t1:dimE}(iii)]
It is a straightforward application of Lemma~\ref{l:DF1}, where we take $\lambda=\log M$.
\end{proof}

\section{The proof of Theorem \ref{t:!bdim}}
\label{S-5}
In this section, we reformulate Theorem~\ref{t:!bdim} as Theorem~\ref{thm-5.1}, stating that  a concretely constructed Cantor set ${\bf E}^*$  is the attractor of a bi-Lipschitz IFS, and has all distinct dimensions. Moreover,  it provides  the precise values of these dimensions of ${\bf E}^*$.

 We begin with the construction of ${\bf E}^*$.  Fix $\beta\in (0,1)$ and $M\geq 100$. Let $(N_k)_{k=1}^\infty$ be a strictly increasing sequence of natural numbers such that $N_1=1$ and
\begin{equation}
\label{e:Nk}
\lim_{k\to \infty}\frac{N_{k+1}}{N_k}=\infty.
\end{equation}
Define a sequence $(M_i)_{i=1}^\infty$ by
\begin{equation}
\label{e-Nk}
M_i=\left\{
\begin{array}{ll}
M,&\mbox{ if }i\in [N_{2k-1}, N_{2k}) \mbox{ for some } k,\\
M+1, &\mbox{ if }i\in [N_{2k}, N_{2k+1}) \mbox{ for some } k.
\end{array}
\right.
\end{equation}
Let ${\bf E}^{\ast}$ be the Cantor set corresponding to the length function $\ell^{\ast}\colon \Omega^*\to (0,1]$ defined by
\begin{equation}
\label{e-ell*}
\ell^{\ast}(\omega_1\ldots\omega_n) = \frac{a_{\omega_1}\cdots a_{\omega_{\lfloor \beta n\rfloor}}}{(a_{\omega_1}\cdots a_{\omega_n})^{\beta}} \cdot \left(\prod_{i=1}^n M_i\right)^{-1},
\end{equation}
where $a_0=1$ and $a_1=2$. That is, ${\bf E}^{\ast}$ is the unique Cantor set whose generating intervals $\{I_\omega\}_{\omega\in\Omega^*}$  satisfy $|I_\omega|=\ell^*(\omega)$. Clearly, ${\bf E}^{\ast}$ depends on the choices of $\beta$, $M$, and $(N_k)_{k=1}^\infty$.

Let ${\bf E}_{\beta,M}$ and ${\bf E}_{\beta,M+1}$ be the Cantor sets corresponding to the length functions $\ell_M$ and $\ell_{M+1}$, respectively, where $\ell_{M}$ and $\ell_{M+1}$ are defined by
\begin{equation}
\label{e-ell12}
\begin{split}
&\ell_{M}(\omega_1\ldots\omega_n) = \frac{a_{\omega_1}\cdots a_{\omega_{\lfloor\beta n\rfloor}}}{(a_{\omega_1}\cdots a_{\omega_n})^{\beta}} \cdot M^{-n}, \\
&\ell_{M+1}(\omega_1\ldots\omega_n) = \frac{a_{\omega_1}\cdots a_{\omega_{\lfloor\beta n\rfloor}}}{(a_{\omega_1}\cdots a_{\omega_n})^{\beta}} \cdot (M+1)^{-n}.
\end{split}
\end{equation}

Recall that $H(\cdot)$ and $D(\cdot)$ are defined respectively as in \eqref{e-entropy} and \eqref{eq:DF}.   The main result of this section is the following.

\begin{thm}
\label{thm-5.1}
 The Cantor set ${\bf E}^*$ is the attractor of a bi-Lipschitz IFS on $\R$ and has distinct lower, Hausdorff, lower box, upper box and Assouad dimensions, which are given precisely by
$$
\hdim {\bf E}^*=\frac{\log 2}{\log (M+1)},\quad \lbdim {\bf E}^*=D(M+1),\quad \ubdim {\bf E}^*=D(M),
$$
and
$$
\adim {\bf E}^* =\sup_{p\in [0,1]}\frac{H(p)}{\log M-\beta(1-p)\log 2},\quad \ldim {\bf E}^*=\sup_{p\in [0,1]}\frac{H(p)}{\log (M+1)+\beta p\log 2}.
$$
\end{thm}

The proof of the above result will be presented in the following subsections by verifying ${\bf E}^*$ as the attractor of a bi-Lipschitz IFS, and determining the various notions of dimension of ${\bf E}^*$ separately.

\subsection{${\bf E}^*$ as the attractor of a bi-Lipschitz IFS} The main result of this subsection is the following.
\begin{prop}
\label{prop-Attractor}
The collection $\Phi=\{\phi_0,\phi_1\}$, where each map $\phi_i$ is defined via the Cantor construction of ${\bf E}^*$ as in \eqref{eq:IFSmap}, forms a bi-Lipschitz IFS whose attractor is ${\bf E}^*$.
\label{prop-attractor}
\end{prop}
\begin{proof}
This is achieved  by  a minor modification of the proof of Proposition~\ref{p:BLIFS}. To avoid repetition,  we only highlight the main differences below.

Instead of \eqref{eq-ratios-3}, we have \begin{equation} \label{eq-ratios-3*}
 \frac{1}{2 M+2} \le  \frac{|I_{i\omega}|}{|I_w|} \le \frac{2}{M} \quad \mbox{ and }\quad \frac{1}{2 M+2} \le  \frac{|I_{\omega i}|}{|I_w|} \le \frac{2}{M}\quad \end{equation}
 for all $\omega\in \Omega^{\ast}$ and $i\in \{0,1\}$, from which we can derive directly
  \begin{equation*}
			\inf_{\omega\in\Omega^*}\min\biggl\{ \frac{|I_{0\omega}|}{|I_\omega|}, \frac{|I_{1\omega}|}{|I_\omega|} \biggr\} \ge \frac{1}{2M+2},\quad
\sup_{\omega\in\Omega^*}\max\biggl\{ \frac{|I_{0\omega}|}{|I_\omega|}, \frac{|I_{1\omega}|}{|I_\omega|} \biggr\} \le \frac{2}{M},
		\end{equation*}
 and
 	\begin{equation*}
				\inf_{\omega\in\Omega^*} \min \biggl\{ \frac{|G_{0\omega}|}{|G_\omega|}, \frac{|G_{1\omega}|}{|G_\omega|} \biggr\} \ge \frac{1}{2M+2} \cdot \frac{M-4}{M}, \quad
			\sup_{\omega\in\Omega^*} \max \biggl\{ \frac{|G_{0\omega}|}{|G_\omega|}, \frac{|G_{1\omega}|}{|G_\omega|} \biggr\} \le\frac{2}{M-4}.
	\end{equation*}
Then by Theorem~\ref{thm:bilip}, $\Phi$ is a bi-Lipschitz IFS.
\end{proof}
\subsection{Hausdorff dimension of ${\bf E}^*$.}
Here, we prove the following proposition.
\begin{prop}
\label{prop-Hausdorff}
$\displaystyle \hdim {\bf E}^*=\frac{\log 2}{\log (M+1)}$.
\end{prop}

We start from the following lemma, which is analogous to Lemma~\ref{e-Mcmullen}.

\begin{lem}For all $\omega= (\omega_n)_{n=1}^{\infty}\in \Omega^{\N}$,
\begin{equation}
\label{e-Mcmullen'}
	\limsup_{n\to\infty}\biggl( \frac{(a_{\omega_1}\dotsm a_{\omega_n})^\beta}{a_{\omega_1}\dotsm a_{\omega_{\lfloor\beta n\rfloor}}}\prod_{i=1}^n M_i \biggr)^{1/n} \ge M+1.
	\end{equation}
Moreover, for every Bernoulli product measure $\nu$ on $\Omega^{\N}$,
\[
\limsup_{n\to\infty}\biggl( \frac{(a_{\omega_1}\dotsm a_{\omega_n})^\beta}{a_{\omega_1}\dotsm a_{\omega_{\lfloor\beta n\rfloor}}}\prod_{i=1}^n M_i \biggr)^{1/n} = M+1
\]
for $\nu$-a.e.~$\omega\in \Omega^{\N}$.
\end{lem}
\begin{proof}
We first prove \eqref{e-Mcmullen'} by adopting the idea used in the proof of  \cite[Lemma 4]{McMul84}.  Suppose on the contrary that \eqref{e-Mcmullen'} does not hold for some $\omega= (\omega_n)_{n=1}^{\infty}\in \Omega^{\N}$. Then there exist $\epsilon\in (0,1)$ and $n_0\in \N$  such that
\begin{equation*}
\biggl( \frac{(a_{\omega_1}\dotsm a_{\omega_n})^\beta}{a_{\omega_1}\dotsm a_{\omega_{\lfloor\beta n\rfloor}}}\prod_{i=1}^n M_i \biggr)^{1/n} \leq (M+1)^{1-\epsilon}\qquad \mbox{ for all }n\geq n_0.
	\end{equation*}
Writing $f(n) \coloneqq (1/n)\log (a_{\omega_1}\dotsm a_{\omega_n})$ and taking the logarithm, we obtain
$$
\beta f(n)-\frac{\lfloor \beta n\rfloor}{n} f(\lfloor\beta n\rfloor)\leq (1-\epsilon)\log (M+1)-\frac{1}{n}\sum_{i=1}^n \log M_i \qquad \mbox{ for all }n\geq n_0.
$$
Since $f(n)\in [0, \log 2]$ is uniformly bounded and $\lfloor\beta n\rfloor/n\to \beta$ as $n\to \infty$, there exists $n_1>n_0$ such that
\begin{equation}
\label{e-fnn}
\beta (f(n)-f(\lfloor\beta n\rfloor))\leq \left(1-\frac{\epsilon}{2}\right)\log (M+1)-\frac{1}{n}\sum_{i=1}^n \log M_i \qquad \mbox{ for all }n\geq n_1.
\end{equation}

Choose $p\in \N$ such that
\begin{equation}
\label{e-rangep}
p\epsilon \log (M+1)>4\log 2.
\end{equation}
Then choose a sufficiently large integer $k$ such that
$$
N_{2k}>n_1   \quad \mbox{ and } \quad \frac{N_{2k+1}}{N_{2k}}\geq \left(\frac{2}{\beta}\right)^p \left(\frac{4}{\epsilon}\right).
$$
By the definition of the sequence $(M_i)$, we see that for $4N_{2k}/\epsilon<n\leq N_{2k+1}$,
$$
\frac{1}{n}\sum_{i=1}^{n}\log M_i\geq \frac{1}{n}\sum_{i=N_{2k}+1}^{n}\log M_i=\left(1-\frac{N_{2k}}{n}\right)\log (M+1)\geq \left(1-\frac{\epsilon}{4}\right)\log (M+1).
$$
Plugging this into \eqref{e-fnn} yields
\begin{equation}
\label{e-fnn1}
\beta (f(n)-f(\lfloor\beta n\rfloor))\leq -\frac{\epsilon}{4}\log (M+1)\qquad \mbox{ for all }\;\frac{4N_{2k}}{\epsilon}<n\leq N_{2k+1}.
\end{equation}

Define $\tau\colon \N\to \N\cup \{0\}$ by $\tau(n)=\lfloor\beta n\rfloor$. Clearly, $\tau(n)\geq \beta n-1\geq (\beta/2)n$ for $n\geq 2/\beta$. Since
$$
\frac{N_{2k+1}}{N_{2k}}\geq \left(\frac{2}{\beta}\right)^p \left(\frac{4}{\epsilon}\right),$$
it follows that $$
\frac{4N_{2k}}{\epsilon}\leq \tau^{i}(N_{2k+1})\leq N_{2k+1},\quad i=1,2,\ldots, p,
$$
where $\tau^i=\underbrace{\tau\circ \cdots \circ \tau}_{i}$ denotes the $i$-th iterate of $\tau$. Hence
\begin{align*}
f(N_{2k+1})-f(\tau^p(N_{2k+1}))&= \sum_{i=0}^{p-1} \left(f(\tau^i(N_{2k+1}))-f(\tau^{i+1}(N_{2k+1}))\right)\\
&\leq -\frac{p\epsilon}{4\beta}\log (M+1)\qquad \qquad \mbox{(by \eqref{e-fnn1})}\\
&<-\log 2\qquad \qquad \qquad \qquad\; \mbox{(by \eqref{e-rangep})},
\end{align*}
which contradicts the fact that  $f(n)\in [0,\log 2]$ for all $n$. Hence \eqref{e-Mcmullen'} always holds.

Next, we prove the second part of the lemma. By the Birkhoff ergodic theorem (see, e.g.,~\cite{Walters}),
$$
\lim_{n\to\infty}\biggl( \frac{(a_{\omega_1}\dotsm a_{\omega_n})^\beta}{a_{\omega_1}\dotsm a_{\omega_{\lfloor\beta n\rfloor}}} \biggr)^{1/n}=1\qquad \mbox{ for $\nu$-a.e.~}(\omega_n)_{n=1}^\infty\in \Omega^\N.
$$
Meanwhile, by the construction of $(M_i)$,
$$
\limsup_{n\to\infty}\left(\prod_{i=1}^n M_i\right)^{1/n}=M+1.
$$
The second part of the lemma then follows by combining the above two limits.
\end{proof}

\begin{proof}[Proof of Proposition \ref{prop-Hausdorff}]
It is done by a routine modification of the proof of Theorem~\ref{t1:dimE}(i), in which we simply use Lemma~\ref{e-Mcmullen'} to replace Lemma~\ref{e-Mcmullen}.
\end{proof}

\subsection{Lower and upper box dimensions of ${\bf E}^*$.}
Here we prove the following.
\begin{prop}
\label{prop-box}
$\lbdim {\bf E}^*=D(M+1)$, and\ $\ubdim {\bf E}^*=D(M)$.
\end{prop}
\begin{proof}
We first prove that
\begin{equation}
\label{e-relation}
\bdim {\bf E}_{\beta, M+1}\leq  \lbdim {\bf E}^*,\qquad \ubdim {\bf E}^*\leq \bdim {\bf E}_{\beta, M}.
\end{equation}
To show this, define for $0<\rho<1$,
\begin{equation}
\label{e-Lambda*M}
\begin{split}
\Lambda^*(\rho)&=\{\omega\in \Omega^*\colon \ell^*(\omega)\leq \rho<\ell^*(\omega^{-})\},\\
\Lambda_{M}(\rho)&=\{\omega\in \Omega^*\colon \ell_M(\omega)\leq \rho<\ell_M(\omega^{-})\},\\
\Lambda_{M+1}(\rho)&=\{\omega\in \Omega^*\colon \ell_{M+1}(\omega)\leq \rho<\ell_{M+1}(\omega^{-})\}.
\end{split}
\end{equation}
By Remark~\ref{rem-4.4},
\begin{equation}
\label{e-MM+1}
\bdim {\bf E}_{\beta, M}=\lim_{\rho\to 0}\frac{\log (\#\Lambda_{M}(\rho))}{\log (1/\rho)},\quad \bdim {\bf E}_{\beta, M+1}=\lim_{\rho\to 0}\frac{\log (\#\Lambda_{M+1}(\rho))}{\log (1/\rho)}.
\end{equation}
Meanwhile, an argument similar to the proof of  Lemma~\ref{lem-4.7}(i) shows that
\begin{equation}
\label{e-E*}
\lbdim {\bf E}^*=\liminf_{\rho\to 0}\frac{\log (\#\Lambda^*(\rho))}{\log (1/\rho)},\qquad \ubdim {\bf E}^*=\limsup_{\rho\to 0}\frac{\log (\#\Lambda^*(\rho))}{\log (1/\rho)}.
\end{equation}
Hence to prove \eqref{e-relation}, it suffices to show that
\begin{equation}
\label{e-claimell}
\#\Lambda_{M+1}(\rho)\leq \#\Lambda^*(\rho)\leq \#\Lambda_{M}(\rho).
\end{equation}

 To avoid repetition, we prove only the first inequality in \eqref{e-claimell}; the second follows by a similar argument.  For each $\omega\in \Lambda^*(\rho)$, since
 $$
 \ell_{M+1}(\omega)\leq \ell^*(\omega)\leq \rho,
 $$
 there exists a unique element $\widetilde{\omega}\in \Lambda_{M+1}(\rho)$ which is a prefix of $\omega$. Hence, to show that
 $\#\Lambda_{M+1}(\rho)\leq \#\Lambda^*(\rho),$ it suffices to show that the mapping $\omega\mapsto \widetilde{\omega}$, from $\Lambda^*(\rho)$ to $\Lambda_{M+1}(\rho)$, is onto. Notice that
 $$
 \bigcup_{\omega\in \Lambda^*(\rho)}[\widetilde{\omega}]\supset  \bigcup_{\omega\in \Lambda^*(\rho)}[\omega]=\Omega^\N,
 $$
and that $\{[u]\colon u\in \Lambda_{M+1}(\rho)\}$ is a partition of $\Omega^\N$. It follows that the mapping $\omega\mapsto \widetilde{\omega}$ is surjective.  This proves \eqref{e-claimell}, and therefore \eqref{e-relation}.

In the remaining part of the proof, we show that
\begin{equation}
\label{e-relation*}
\bdim {\bf E}_{\beta, M+1}\geq  \lbdim {\bf E}^*,\qquad  \ubdim {\bf E}^*\geq \bdim {\bf E}_{\beta, M}.
\end{equation}
To this end, by \eqref{e-MM+1} and \eqref{e-E*}, it suffices to show that for all $\epsilon, \delta\in (0,1)$, there exist $\rho_1, \rho_2\in (0,\delta)$ such that
\begin{equation}
\label{e-lambda*}
\# \Lambda_{M+1}(\rho_1)\geq \#\Lambda^*(\rho_1^{1-\epsilon}),\qquad \# \Lambda^*(\rho_2)\geq \#\Lambda_{M}(\rho_2^{1-\epsilon}).
\end{equation}
Without loss of generality, below we will only show that for given $\epsilon,\delta\in (0,1)$, the first inequality in \eqref{e-lambda*} holds for some $\rho_1\in (0,\delta)$. The proof of the second inequality is similar.

Recall that by Lemma~\ref{lem-4.7}(ii), for each $\rho\in (0,1)$, if $\Omega^n\cap \Lambda_{M+1}(\rho)\neq \varnothing$, then
\begin{equation}
\label{e-C1C2}
C_1\log (1/\rho)\leq n\leq C_2\log (1/\rho),
\end{equation}
where  $$C_1=\frac{1}{\log (2M+2)}\quad \mbox{ and }\quad C_2=\frac{2}{\log ((M+1)/2)}.
$$

Let $\epsilon, \delta\in (0,1)$. Take a sufficiently large integer $k$ such that
\begin{equation}
\label{e-N2k}
\frac{N_{2k}}{N_{2k+1}}\leq \min\left\{\frac{\epsilon }{2C_2 \log (M+1)}, \;\frac{C_1}{2C_2}\right\}\quad \mbox{ and }\quad \exp\left(-\frac{N_{2k+1}}{2 C_2}\right)<\delta,
\end{equation}
and set $$\rho_1=\exp\left(-\frac{N_{2k+1}}{2 C_2}\right).$$
Clearly $\rho_1\in (0,\delta)$. In what follows we show that  $\#\Lambda_{M+1}(\rho_1)\geq \# \Lambda^*(\rho_1^{1-\epsilon})$.

 We first claim that
\begin{equation}
\label{e-ell*omega}
\ell^*(\omega)\leq \rho_1^{1-\epsilon}\qquad \mbox{ for each }\omega\in \Lambda_{M+1}(\rho_1).
\end{equation}
To see this, let $\omega=\omega_1\ldots \omega_n\in  \Lambda_{M+1}(\rho_1)$. Then $\ell_{M+1}(\omega)\leq \rho_1$.  By the first inequality in \eqref{e-N2k}, and \eqref{e-C1C2} (in which we take $\rho=\rho_1$), we obtain
$$
\frac{C_1\log(M+1)}{\epsilon}N_{2k}\leq \frac{C_1}{2C_2}\cdot N_{2k+1}\leq n\leq \frac{N_{2k+1}}{2}.
$$
This implies, in particular, that
\begin{equation}
(M+1)^{N_{2k}}\leq \exp\left(\frac{\epsilon}{2C_2}\cdot N_{2k+1}\right)=\rho_1^{-\epsilon}
\end{equation}
and that
$$
N_{2k}<n\leq \frac{N_{2k+1}}{2}.
$$
Hence by definition,  $M_i=M+1$ for $N_{2k}<i\leq n$. It follows that
$$
\ell^*(\omega)=\ell_{M+1}(\omega)\cdot \frac{(M+1)^n}{\prod_{i=1}^nM_i}\leq \ell_{M+1}(\omega)\cdot (M+1)^{N_{2k}}\leq \rho_1\cdot \rho_1^{-\epsilon}=\rho_1^{1-\epsilon}.
$$
This proves  \eqref{e-ell*omega}.

Now by \eqref{e-ell*omega}, we see that for each $\omega\in \Lambda_{M+1}(\rho_1)$, there exists a unique $\omega'\in \Lambda^*(\rho_1^{1-\epsilon})$ which is a prefix of $\omega$.  Since
$$
\bigcup_{\omega\in \Lambda_{M+1}(\rho_1)}[\omega']\supset \bigcup_{\omega\in \Lambda_{M+1}(\rho_1)}[\omega]=\Omega^\N
$$
and $\{[u]\colon u\in \Lambda^*(\rho_1^{1-\epsilon})\}$ is a partition of $\Omega^\N$, we see that the mapping $\omega\mapsto \omega'$, $\Lambda_{M+1}(\rho_1)\to \Lambda^*(\rho_1^{1-\epsilon})$, is onto. Therefore, $\#\Lambda_{M+1}(\rho_1)\geq \# \Lambda^*(\rho_1^{1-\epsilon})$, as desired. This proves \eqref{e-relation*}.

Combining \eqref{e-relation} and \eqref{e-relation*}, we obtain that $\lbdim {\bf E}^*=\bdim {\bf E}_{\beta, M+1}$ and $\ubdim {\bf E}^*=\bdim {\bf E}_{\beta, M}$. The proposition then follows from Theorem~\ref{t1:dimE}(ii).
\end{proof}

 \subsection{Assouad and lower dimensions of ${\bf E}^*$.}
 \label{S-ALdim}
  The main result of this subsection is the following.

\begin{prop}
\label{prop-lowerassouad}
We have
\begin{equation}\label{eq:adim}
		\adim\mathbf{E}^* ~=~ \max_{p\in[0,1]}\frac{H(p)}{\log M - \beta(1-p)\log2},
	\end{equation}
and
\begin{equation}\label{eq:ldim}
		\ldim\mathbf{E}^* ~=~ \max_{p\in[0,1]}\frac{H(p)}{\log (M+1) + \beta p \log2}.
	\end{equation}		
\end{prop}
Before proving the above proposition, we first introduce some notation. For $u\in \Omega^*$ and $0<\rho<1$, we define
\begin{equation}
\label{e-Lambda*u}
\Lambda^{*,u}(\rho)=\left\{\omega\in \Omega^*\colon \frac{\ell^*(u\omega)}{\ell^*(u)}\leq \rho<\frac{\ell^*(u\omega^{-})}{\ell^*(u)}\right\}.
\end{equation}

Recall that $\{I_\omega\colon \omega\in \Omega^*\}$ denotes the collection of generating intervals of the Cantor set ${\bf E}^*$. The following lemma simply follows from the Cantor construction of ${\bf E}^*$ and the easily verified fact that $\ell^*(\omega)/\ell^*(\omega^-)\geq \frac{1}{2M+2}$ for each $\omega$.

\begin{lem}
\begin{itemize}
\item[(i)] For every generating interval $I_\omega$ of ${\bf E}^*$, the two gap intervals adjacent to $I_\omega$  have lengths greater than $|I_\omega|$.
\item[(ii)] Let $u\in \Omega^*$ and $\rho\in (0,1)$. Then the intervals $I_{u\omega}$, for $\omega\in \Lambda^{*,u}(\rho)$, are disjoint subintervals of $I_u$,  each having length between $\frac{1}{2M+2}\rho \ell^*(u)$ and $\rho \ell^*(u)$. Moreover,
$$
I_u\cap {\bf E}^*=\left(\bigcup_{\omega\in \Lambda^{*,u}(\rho)} I_{u\omega}\right)\cap {\bf E}^*.
$$
\end{itemize}
\end{lem}

The following result follows from the above lemma and the definitions of  Assouad and lower dimensions by a routine argument; we omit the details.

\begin{lem}
\label{lem-5.4}
 We have
$$
\adim{{\bf E}^*}=\limsup_{\rho\to 0}\sup_{u\in \Omega^*} \frac{\log(\#\Lambda^{*,u}(\rho))}{\log(1/\rho)},
$$
and
$$
\ldim{{\bf E}^*}=\liminf_{\rho\to 0}\inf_{u\in \Omega^*} \frac{\log(\#\Lambda^{*,u}(\rho))}{\log(1/\rho)}.
$$

\end{lem}

Recall that $a_0=1$ and $a_1=2$. For $I=i_1\ldots i_m\in \Omega^m$, we write for brevity $a_I=\prod_{k=1}^ma_{i_k}$. For a word $\omega\in \Omega^*$, let  ${\mathsf L}_1(\omega)$ denote the number of times the letter $1$ appears in  $\omega$. Moreover, for $\omega=\omega_1\ldots \omega_k\in \Omega^*$, write $\omega|n \coloneqq \omega_1\ldots \omega_n$ for $n\leq k$.
\begin{lem}
Let $u=u_1\ldots u_k\in \Omega^k$ and $\omega=\omega_1\ldots \omega_n\in \Omega^n$. Then
\begin{equation}
\label{e-lengthuw}
\frac{\ell^*(u\omega)}{\ell^*(u)}=a_\omega^{-\beta}\cdot a_{v|(\lfloor\beta(n+k)\rfloor-\lfloor\beta k\rfloor)}\cdot \frac{1}{\prod_{i=1}^n M_{k+i}},
\end{equation}
where $\nu \coloneqq u_{\lfloor\beta k\rfloor+1}\ldots u_k \omega_1\ldots \omega_n$. Consequently,
\begin{equation}
\label{e-lengthuw1}
 2^{-\beta {\mathsf L}_1(\omega)}(M+1)^{-n}\leq \frac{\ell^*(u\omega)}{\ell^*(u)}\leq 2^{1+\beta n-\beta {\mathsf L}_1(\omega)}M^{-n}.
\end{equation}
\end{lem}
\begin{proof}
Clearly,  \eqref{e-lengthuw1} follows from  \eqref{e-lengthuw}, using the inequality $\lfloor\beta(n+k)\rfloor-\lfloor\beta k\rfloor\leq 1+\beta n$.  To see \eqref{e-lengthuw}, by the definition of $\ell^*(\cdot)$ (see \eqref{e-ell*}), we obtain
\begin{align*}
\frac{\ell^*(u\omega)}{\ell^*(u)}&=\frac{a_{(u\omega)|\lfloor\beta(n+k)\rfloor}}{a_{u\omega}^\beta}\cdot \frac{a_u^\beta}{a_{u|\lfloor\beta k\rfloor}}\cdot \frac{1}{\prod_{i=1}^n M_{k+i}}\\
&=a_\omega^{-\beta}\cdot a_{v|(\lfloor\beta(n+k)\rfloor-\lfloor\beta k\rfloor)}\cdot \frac{1}{\prod_{i=1}^n M_{k+i}},
\end{align*}
as desired.
\end{proof}
As a direct consequence of the above lemma, we have the following.

\begin{cor}
\label{cor-5.6}
 Let $u=u_1\ldots u_k\in \Omega^k$ with $u=0^k$ or $1^k$, and $\omega=\omega_1\ldots \omega_n\in \Omega^n$. Suppose that $n\leq \frac{(1-\beta)k}{\beta}$. Then
    $$
\frac{\ell^*(u\omega)}{\ell^*(u)}=\left\{
\begin{array}{ll}
2^{-\beta {\mathsf L}_1(\omega)}\cdot\left(\prod_{i=1}^n M_{k+i}\right)^{-1}, &\mbox{ if }u=0^k,\\[\medskipamount]
2^{-\beta {\mathsf L}_1(\omega)+\lfloor\beta(n+k)\rfloor-\lfloor\beta k\rfloor}\cdot \left(\prod_{i=1}^n M_{k+i}\right)^{-1},&\mbox{ if }u=1^k.
\end{array}
\right.
    $$
\end{cor}

We still need the following result, whose part (i) is analogous to Lemma~\ref{lem-4.7}(ii).
\begin{lem}
\label{lem-5.6}
Let $u\in \Omega^*$ and $0<\rho<1$. Then the following properties hold.
\begin{itemize}
\item[(i)] If $\Omega^n\cap \Lambda^{*,u}(\rho)\neq \varnothing$, then
$$
\frac{\log (1/\rho)}{\log (2M+2)}\leq n\leq  \frac{2\log (1/\rho)}{\log (M/2)}
$$
provided that $0<\rho<(2M+2)^{-1}$.
\item[(ii)] For every $\omega\in \Omega^n\cap \Lambda^{*,u}(\rho)$,
$$
 \frac{1}{4M+4}\cdot M^n 2^{\beta ({\mathsf L}_1(\omega)-n)}\leq \frac{1}{\rho}\leq 2^{\beta {\mathsf L}_1(\omega)}(M+1)^n.
$$
\end{itemize}
\end{lem}
\begin{proof}
By the definition of the length function $\ell^*$ (see \eqref{e-ell*}) and an argument similar to the proof of \eqref{eq-ratios-3}, we have
$$
\frac{1}{2M+2}\leq \frac{\ell^*(\omega i)}{\ell^*(\omega)}\leq \frac{2}{M} \quad \mbox{ for all }\omega\in \Omega^* \mbox{ and }i\in \{0,1\}.
$$
It follows that
$$
(2M+2)^{-n}\leq \frac{\ell^*(u\omega)}{\ell^*(u)}\leq (M/2)^{-n} \qquad \mbox{ for all }\omega\in \Omega^n
$$
and that
\begin{equation}
\label{e-2M+2}
\frac{\rho}{2M+2}<\frac{1}{2M+1}\cdot \frac{\ell^*(u\omega^{-})}{\ell^*(u)}\leq \frac{\ell^*(u\omega)}{\ell^*(u)}\leq \rho\quad \mbox{ if }\omega\in \Omega^n\cap \Lambda^{*,u}(\rho).
\end{equation}
Hence if $\omega\in \Omega^n\cap \Lambda^{*,u}(\rho)$, then
$$
(2M+2)^{-n}\leq \frac{\ell^*(u\omega)}{\ell^*(u)}\leq \rho\quad\mbox{ and }\quad  \frac{\rho}{2M+2}<\frac{\ell^*(u\omega)}{\ell^*(u)}\leq (M/2)^{-n},
$$
which implies
$$
\frac{\log (1/\rho)}{\log (2M+2)}\leq n\leq \frac{\log ((2M+2)\rho^{-1})}{\log (M/2)}\leq \frac{2\log(1/\rho)}{\log (M/2)},
$$
provided that  $0<\rho\leq (2M+2)^{-1}$. This proves (i).

Part (ii) follows directly from \eqref{e-2M+2} and \eqref{e-lengthuw1}.
  \end{proof}

\begin{proof}[Proof of Proposition \ref{prop-lowerassouad}] We divide the proof into 4 steps.

{\it Step 1: upper bound for $\adim{\bf E}^*$}. For our convenience, write
\begin{equation}
\label{e-Da}
D_a \coloneqq \max_{p\in[0,1]}\frac{H(p)}{\log M - \beta(1-p)\log2}.
\end{equation}
Let $c$ be the constant given as in Lemma~\ref{lem-4.3}.  We will show that for each $u\in \Omega^*$ and $\rho\in (0,\frac{1}{2M})$,
\begin{equation}
\label{e-upperassouad1}
\#\Lambda^{*,u}(\rho)\leq 2\cdot \left(\frac{2\log(1/\rho)}{\log (M/2)}\right)^{2+c}  \cdot \left(\frac{4 M+4}{\rho}\right)^{D_a},
\end{equation}
which, together with Lemma~\ref{lem-5.4}, implies that $\adim {\bf E}^*\leq D_a$.

To prove \eqref{e-upperassouad1}, fix $u\in \Omega^*$ and $\rho\in (0,\frac{1}{2M})$. Let ${\mathcal N}_{u,\rho}$ denote the set of integers $n$ such that $\Omega^n\cap \Lambda^{*,u}(\rho)\neq \varnothing$. By Lemma~\ref{lem-5.6}(i),
\begin{equation}
\label{e-nLambda}
\frac{\log (1/\rho)}{\log (2M+2)}\leq n\leq  \frac{2\log (1/\rho)}{\log (M/2)}\quad \mbox{ for each }n\in {\mathcal N}_{u,\rho}.
\end{equation}
Notice that for  $n\in {\mathcal N}_{u,\rho}$ and $0\leq j\leq n$, if the set
$$
\{\omega\in \Omega^n\cap \Lambda^{*,u}(\rho)\colon\; {\mathsf L}_1(\omega)=j\}
$$
is non-empty, then by Lemma~\ref{lem-5.6}(ii),
\begin{equation}
\label{e-rho4M}
 \frac{1}{4M+4}\cdot M^n 2^{\beta (j-n)}\leq \frac{1}{\rho}
\end{equation}
and hence
\begin{align*}
\#&\{\omega\in \Omega^n\cap \Lambda^{*,u}(\rho)\colon\; {\mathsf L}_1(\omega)=j\}\\
\mbox{}&\leq \binom{n}{j}\\
&\leq n^c \exp\left(nH\left(\frac{j}{n}\right)\right)\qquad\qquad\qquad\quad\qquad\qquad\mbox{(by \eqref{e-4.5})}\\\\
&\leq n^c \exp\left(D_an\left(\log M-\beta(1-j/n)\log 2\right)\right)\qquad\mbox{(by the definition of $D_a$)}\\
&=n^c \left(M^n 2^{\beta (j-n)}\right)^{D_a}\\
&\leq n^c\left(\frac{4M+4}{\rho} \right)^{D_a}\qquad\qquad\qquad\quad\quad \qquad\qquad \mbox{(by \eqref{e-rho4M})}.
\end{align*}
It follows that
\begin{align*}
\#\Lambda^{*,u}(\rho)&= \sum_{n\in {\mathcal N}_{u,\rho}}\sum_{j=0}^n\#\{\omega\in \Omega^n\cap \Lambda^{*,u}(\rho)\colon\; {\mathsf L}_1(\omega)=j\}\\
&\leq \sum_{n\in {\mathcal N}_{u,\rho}} (n+1)n^c\left(\frac{4M+4}{\rho} \right)^{D_a}\\
&\leq 2\cdot\left(\frac{2\log (1/\rho)}{\log (M/2)}\right)^{2+c}\cdot\left(\frac{4M+4}{\rho} \right)^{D_a}\qquad \mbox{(by \eqref{e-nLambda})},
\end{align*}
which proves \eqref{e-upperassouad1}.
\medskip

{\it Step 2: lower bound of $\adim {\bf E}^*$}. Recall that $D_a$ is defined as in \eqref{e-Da}. By Lemma~\ref{lem-5.4}, to prove that $\adim {\bf E}^*\geq D_a$, it suffices to show that for all $\epsilon>0$, there exist $u\in \Omega^*$ and $0<\rho<\epsilon$ such that
$$
\#\Lambda^{*,u}(\rho)\geq (1/\rho)^{D_a-\epsilon}.
$$

To this end, let $\epsilon>0$, and let $p_0\in [0,1]$ be a point at which the maximum on the right-hand side of \eqref{e-Da} is attained. Then we can take a sufficiently large $n\in \N$ such that
\begin{equation}
\label{e-Mnepsilon}
M^n 2^{\beta \lfloor p_0n\rfloor-\beta n+1}\geq \epsilon^{-1}
\end{equation}
and
\begin{equation}
\label{e-entropy1}
\frac{nH\left(\frac{\lfloor p_0n\rfloor}{n}\right)-c\log n } {\log\left( M^n 2^{\beta \lfloor p_0n\rfloor-\beta n+1}\right) }\geq D_a-\epsilon.
\end{equation}

Next we pick a large positive integer $j$ such that
$$
\frac{(1-\beta)}{\beta}N_{2j-1}>n,\quad \mbox{ and }\quad N_{2j}-N_{2j-1}>n.
$$
Set $u=1^{N_{2j-1}}\in \Omega^{N_{2j-1}}$. By Corollary~\ref{cor-5.6} and the definition of $(M_i)$ (see \eqref{e-Nk}),
$$
\frac{\ell^*(u\omega)}{\ell^*(u)}=2^{-\beta {\mathsf L}_1(\omega)+\lfloor\beta(n+k)\rfloor-\lfloor\beta k\rfloor}\cdot M^{-n}\quad \mbox{ for all } \omega\in \Omega^n.
$$
Set $\displaystyle\rho=2^{-\beta \lfloor p_0n\rfloor+\lfloor\beta(n+k)\rfloor-\lfloor\beta k\rfloor}\cdot M^{-n}.$ Then the above equality implies that
$$
\Lambda^{*,u}(\rho)\supset \{\omega\in \Omega^n\colon {\mathsf L}_1(\omega)=\lfloor p_0n\rfloor\}.
$$
Hence by \eqref{e-4.5},
$$
\#\Lambda^{*,u}(\rho)\geq \binom{n}{\lfloor p_0n\rfloor}\geq n^{-c}\exp\left(nH\left(\frac{\lfloor p_0n\rfloor}{n}\right)\right).
$$
Combining this with \eqref{e-entropy1} yields
$$
\#\Lambda^{*,u}(\rho)\geq \left( M^n 2^{\beta \lfloor p_0n\rfloor-\beta n+1}\right) ^{D_a-\epsilon}\geq \left( M^n 2^{\beta \lfloor p_0n\rfloor-\lfloor\beta(n+k)\rfloor+\lfloor\beta k\rfloor}\right) ^{D_a-\epsilon}= \left(\frac{1}{\rho}\right)^{D_a-\epsilon},
$$
where in the second inequality we have used the easily verified fact that $\lfloor\beta(n+k)\rfloor-\lfloor\beta k\rfloor\geq \beta n-1$.
Note that this fact, together with \eqref{e-Mnepsilon}, also implies that $$\rho\leq M^{-n}  2^{-\beta \lfloor p_0n\rfloor+\beta n-1}<\epsilon.$$
This completes the proof of the lower bound of $\adim {\bf E}^*$.

{\it Step 3: upper bound for $\ldim {\bf E}^*$}.  For our convenience, write
\begin{equation}
\label{e-Dl}
D_l \coloneqq \max_{p\in[0,1]}\frac{H(p)}{\log (M+1) + \beta p\log2}.
\end{equation}
To prove that $\ldim {\bf E}^*\leq D_l$, by Lemma~\ref{lem-5.4}, it suffices to show that for all $0<\rho<\frac{1}{2M}$, there exists $u\in \Omega^*$ such that
\begin{equation}
\label{e-upperassouad1*}
\#\Lambda^{*,u}(\rho)\leq 2\cdot \left(\frac{2\log(1/\rho)}{\log (M/2)}\right)^{2+c}  \cdot \left(\frac{2 M+2}{\rho}\right)^{D_a}.
\end{equation}

To this end, fix  $\rho\in (0,\frac{1}{2M})$. Take a large $j\in \N$ such that
\begin{equation}
\label{e-Njj}
\min\left\{\frac{(1-\beta)N_{2j}}{\beta},\; N_{2j+1}-N_{2j}\right\}\geq \frac{2\log (1/\rho)}{\log (M/2)}.
\end{equation}
Then set $u=0^{N_{2j}}\in \Omega^{N_{2j}}$. We show below that \eqref{e-upperassouad1*} holds.

Let ${\mathcal N}_{u,\rho}$ denote the set of integers $n$ such that $\Omega^n\cap \Lambda^{*,u}(\rho)\neq \varnothing$. By Lemma~\ref{lem-5.6}(i) and \eqref{e-Njj}, if $n\in {\mathcal N}_{u,\rho}$, then
\begin{equation}
\label{e-nLambda*}
 n\leq  \frac{2\log (1/\rho)}{\log (M/2)}\leq \min\left\{\frac{(1-\beta)N_{2j}}{\beta},\; N_{2j+1}-N_{2j}\right\}.
\end{equation}
Hence by Corollary~\ref{cor-5.6} and the definition of $(M_i)$, for every $\omega\in \Omega^n\cap \Lambda^{*,u}(\rho)$, we have
$$
2^{-\beta {\mathsf L}_1(\omega)}(M+1)^{-n}=\frac{\ell^*(u\omega)}{\ell^*(u)}\geq \frac{1}{2M+2}\cdot \frac{\ell^*(u\omega^{-})}{\ell^*(u)}> \frac{\rho}{2M+2},
$$
which is equivalent to
$$
\frac{1}{\rho}>\frac{2^{\beta {\mathsf L}_1(\omega)}(M+1)^{n}}{2M+2}.
$$
It follows that for  $n\in {\mathcal N}_{u,\rho}$ and $0\leq j\leq n$, if the set
$$
\{\omega\in \Omega^n\cap \Lambda^{*,u}(\rho)\colon\; {\mathsf L}_1(\omega)=j\}
$$
is non-empty, then
\begin{equation}
\label{e-rho4M*}
 \frac{(M+1)^n 2^{\beta j}}{2M+2}< \frac{1}{\rho}
\end{equation}
and hence
\begin{align*}
\#&\{\omega\in \Omega^n\cap \Lambda^{*,u}(\rho)\colon\; {\mathsf L}_1(\omega)=j\}\\
\mbox{}&\leq \binom{n}{j}\\
&\leq n^c \exp\left(nH\left(\frac{j}{n}\right)\right)\qquad\qquad\qquad\quad\qquad\qquad\mbox{(by \eqref{e-4.5})}\\\\
&\leq n^c \exp\left(D_{\tb l} n\left(\log (M+1)+\beta (j/n)\log 2\right)\right)\qquad\mbox{(by the definition of $D_l$)}\\
&=n^c \left((M+1)^n 2^{\beta j}\right)^{D_l}\\
&\leq n^c\left(\frac{2M+2}{\rho} \right)^{D_l}\qquad\qquad\qquad\quad\quad \qquad\qquad \mbox{(by \eqref{e-rho4M*})}.
\end{align*}
It follows that
\begin{align*}
\#\Lambda^{*,u}(\rho)&= \sum_{n\in {\mathcal N}_{u,\rho}}\sum_{j=0}^n\#\{\omega\in \Omega^n\cap \Lambda^{*,u}(\rho)\colon\; {\mathsf L}_1(\omega)=j\}\\
&\leq \sum_{n\in {\mathcal N}_{u,\rho}} (n+1)n^c\left(\frac{2M+2}{\rho} \right)^{D_a}\\
&\leq 2\cdot\left(\frac{2\log (1/\rho)}{\log (M/2)}\right)^{2+c}\cdot\left(\frac{2M+2}{\rho} \right)^{D_a}\qquad \mbox{(by \eqref{e-nLambda*})},
\end{align*}
which proves \eqref{e-upperassouad1*}.
\medskip

{\it Step 4: lower bound for $\ldim{\bf E}^*$}.  Recall that  $D_l$ is defined as in \eqref{e-Dl}. We will show that  $\ldim {\bf E}^*>D_l-\epsilon$ for all $\epsilon>0$. To this end, let $\epsilon>0$. Let $p_1$ be a point in $[0,1]$ at which the maximum on the right-hand side of \eqref{e-Dl} is attained, and set $p_0=1-p_1$. By continuity, there exists $\delta>0$ such that
\begin{equation}
\label{e-5.33}
\frac{H(p_1)-\delta}{\log (M+1) + \beta (p_1+\delta)\log2}\geq D_l-\epsilon.
\end{equation}

Next let $\eta$ be the Bernoulli product measure on $\Omega^\N$ generated by the probability vector $(p_0, p_1)$. That is, $\eta$ is the unique Borel probability measure on $\Omega^\N$ such that
$$
\eta([i_1\ldots i_n])=p_{i_1}\cdots p_{i_n}
$$
for every cylinder set $[i_1\ldots i_n]$. By the Shannon--McMillan--Breiman theorem and the Birkhoff ergodic theorem (or the law of large numbers), for $\eta$-a.e.~$x\in \Omega^\N$,
$$
\lim_{n\to \infty} \frac{\log \eta([x|n])}{n}=-H(p_1),\qquad \lim_{n\to \infty} \frac{{\mathsf L}_1(x|n)}{n}=p_1,
$$
where $x|n$ denotes the prefix $x_1\ldots x_n$ of $x=(x_i)_{i=1}^\infty$. Hence, by the Egoroff theorem, there exist a Borel set $X\subset \Omega^\N$ with $\eta(X)>1/2$ and $N\in \N$ such that for all $x\in X$ and $n\geq N$,
\begin{equation}
\label{e-5.34}
\eta([x|n])\leq \exp(-nH(p_1)+n\delta),\qquad {\mathsf L}_1(x|n)\leq (p_1+\delta)n.
\end{equation}

Now let $\rho\in (0, 1/(2M+2))$ be sufficiently small so that
$$
\frac{\log (1/\rho)}{\log (2M+2)}>N.
$$
Let $u\in \Omega^*$. By Lemma~\ref{lem-5.6}(i), for $n\in \N$, if $\Omega^n\cap \Lambda^{*,u}(\rho)\neq \varnothing$ (i.e., $n\in \mathcal N_{u,\rho}$), then
\begin{equation}
\frac{\log (1/\rho)}{\log (2M+2)}\leq n\leq \frac{2 \log (1/\rho)}{\log (M/2)}.
\end{equation}
Consequently, $n\geq N$ for all $n\in \mathcal N_{u,\rho}$, and $\displaystyle\#\mathcal N_{u,\rho}\leq \frac{2 \log (1/\rho)}{\log (M/2)}$.

Note that $\bigcup_{\omega\in \Lambda^{*,u}(\rho)}([\omega])=\Omega^\N$. It follows that
$$
\eta\left(\bigcup_{n\in \mathcal N_{u,\rho}}\bigcup_{\omega\in \Omega^n\cap\Lambda^{*,u}(\rho)}([\omega]\cap X) \right)=\eta\left(\bigcup_{\omega\in \Lambda^{*,u}(\rho)}([\omega]\cap X)\right)=\eta(X)>\frac{1}{2}.
$$
Hence there exists $\widetilde{n}\in \mathcal N_{u,\rho}$ such that
\begin{equation}
\label{e-5.36}
\eta\left(\bigcup_{\omega\in \Omega^{\widetilde{n}}\cap\Lambda^{*,u}(\rho)}([\omega]\cap X) \right)\geq \frac{1}{2\cdot\#\mathcal N_{u,\rho}}\geq \frac{\log (M/2)}{4 \log (1/\rho)}\geq \frac{\log (M/2)}{4 \log (2M+2)}\cdot \frac{1}{\widetilde{n}},
\end{equation}
where we have used the proven facts that $\displaystyle\#\mathcal N_{u,\rho}\leq \frac{2 \log (1/\rho)}{\log (M/2)}$ in the second inequality, and $\displaystyle\widetilde{n}\geq \frac{\log (1/\rho)}{\log (2M+2)}$ in the last inequality. Notice that $\widetilde{n}\geq N$. So by \eqref{e-5.34}, for every $\omega\in \Omega^{\widetilde{n}}$ with $[\omega]\cap X\neq \varnothing$,  we have
\begin{equation}
\label{e-5.37}
\eta([\omega])\leq \exp(-\widetilde{n}H(p_1)+\widetilde{n}\delta),\quad\mbox{ and }\quad
\mathsf L_1(\omega)\leq (p_1+\delta)\widetilde{n}.
\end{equation}
It follows that
\begin{align*}
\eta\left(\bigcup_{\omega\in \Omega^{\widetilde{n}}\cap\Lambda^{*,u}(\rho)}([\omega]\cap X) \right)&\leq \#(\Omega^{\widetilde{n}}\cap\Lambda^{*,u}(\rho))\cdot \exp(-\widetilde{n}H(p_1)+\widetilde{n}\delta)\\
&\leq (\# \Lambda^{*,u}(\rho))\cdot \exp(-\widetilde{n}H(p_1)+\widetilde{n}\delta).
\end{align*}
Combining this with \eqref{e-5.36} gives
\begin{equation}
\label{e-5.38}
\# \Lambda^{*,u}(\rho)\geq \frac{\log (M/2)}{4 \log (2M+2)}\cdot \frac{1}{\widetilde{n}}\cdot \exp(\widetilde{n}H(p_1)-\widetilde{n}\delta).
\end{equation}
Meanwhile, taking $\omega\in \Omega^{\widetilde{n}}$ with $[\omega]\cap X\neq \varnothing$, and applying Lemma~\ref{lem-5.6}(ii) and the second inequality in \eqref{e-5.37}, we obtain
$$
\frac{1}{\rho}\leq 2^{\beta {\mathsf L}_1(\omega)}(M+1)^{\widetilde{n}}\leq 2^{\beta (p_1+\delta)\widetilde{n}}(M+1)^{\widetilde{n}}.
$$
This, together with \eqref{e-5.38}, yields
$$
\frac{\log (\# \Lambda^{*,u}(\rho))}{\log (1/\rho)}\geq \frac{\log \left(\frac{\log (M/2)}{4 \log (2M+2)}\cdot \frac{1}{\widetilde{n}}\cdot \exp(\widetilde{n}H(p_1)-\widetilde{n}\delta)\right)}{\widetilde{n} (\log (M+1)+\beta (p_1+\delta)\log 2 )}.
$$
Note that $\displaystyle\widetilde{n}\geq \frac{\log (1/\rho)}{\log (2M+2)}$, so $\widetilde{n}\to \infty$ as $\rho\to 0$.  Hence, by the above inequality, Lemma~\ref{lem-5.4}, and \eqref{e-5.33}, we obtain
$$
\ldim{{\bf E}^*}=\liminf_{\rho\to 0}\inf_{u\in \Omega^*} \frac{\log(\#\Lambda^{*,u}(\rho))}{\log(1/\rho)}\geq D_l-\epsilon,
$$
as desired.
\end{proof}
\subsection{The proof of Theorem \ref{thm-5.1}}
We begin with the following lemma.
\begin{lem}
\label{lem-Last}
The following two inequalities  hold:
\begin{equation}
\label{e-in1}
\max_{p\in[0,1]}\frac{H(p)}{\log (M+1) + \beta p\log2}<\frac{\log 2}{\log (M+1)},
\end{equation}
and
\begin{equation}
\label{e-in2}
\max_{(p,q)\in[0,1]^2} \frac{\beta H(p) + (1-\beta) H(q)}{\log M + \beta(1-\beta)(q-p)\log 2}<\max_{p\in[0,1]}\frac{H(p)}{\log M -\beta (1-p)\log2}.
\end{equation}
\end{lem}
\begin{proof} Inequality \eqref{e-in1}  is straightforward to verify  since the maximum on its left-hand side is attained at some $p>0$, and $H(p)\leq \log 2$.

To see \eqref{e-in2}, let $(p_0,q_0)$ be a point in $[0,1]^2$ at which the maximum on the left-hand side of \eqref{e-in2} is attained.  Since the entropy  function $H$ is concave on $[0,1]$,  by Jensen's inequality,
\begin{equation}
\label{e-entropyH}
\beta H(p_0) + (1-\beta) H(q_0)\leq H(\beta p_0 + (1-\beta) q_0).
\end{equation}
Set $t=\beta p_0 + (1-\beta) q_0$. Clearly, $t\in [0,1]$ and
\begin{align}
\label{e-align}
(1-\beta)(q_0-p_0)+(1-t)=1-p_0\geq 0,
\end{align}
which implies  $\log M + \beta(1-\beta)(q_0-p_0)\log 2\geq \log M-\beta(1-t)\log 2$. Hence
\begin{align}
\label{e-Last}
\frac{\beta H(p_0) + (1-\beta) H(q_0)}{\log M + \beta(1-\beta)(q_0-p_0)\log 2}\leq \frac{ H(t)}{\log M - \beta(1-t)\log 2}.
\end{align}
We claim that the inequality in  \eqref{e-Last} is strict. Otherwise, the inequalities in  \eqref{e-entropyH} and \eqref{e-align} would  all be equalities, which implies that $p_0=q_0$ (by the strict concavity of $H$) and $p_0=1$, leading to a contradiction since the left-hand side of \eqref{e-in2} would then be equal to zero. Hence, the claim is true, and \eqref{e-in2} follows.
\end{proof}
\begin{proof}[Proof of Theorem \ref{thm-5.1}] By Proposition~\ref{prop-Attractor}, ${\bf E}^*$ is the attractor of a bi-Lipschitz IFS.
As shown in Propositions \ref{prop-Hausdorff}, \ref{prop-box} and \ref{prop-lowerassouad}, ${\bf E}^*$ satisfies the desired dimensional formulas.

It remains to show that all these dimensions are distinct.  Since the function $D(\cdot)$ is strictly decreasing, it follows that $D(M+1)<D(M)$, and thus $\lbdim {\bf E}^*<\ubdim {\bf E}^*$. By Lemma~\ref{l:DF1}, $\frac{\log 2}{\log (M+1)}<D(M+1)$, and consequently, $\hdim {\bf E}^*<\lbdim {\bf E}^*$. Finally, the inequalities $\ldim {\bf E}^*<\hdim {\bf E}^*$ and $\ubdim {\bf E}^*<\adim {\bf E}^*$ follow directly from Lemma~\ref{lem-Last}.
\end{proof}
\section{A simpler construction}\label{Sec:simpler}

In this section, we present a simpler construction of a bi-Lipschitz IFS whose attractor has distinct lower and upper box dimensions. However in this case, the Hausdorff dimension cannot be separated from the lower box dimension. Let us first introduce the following.
\begin{defn}
Let $\mathbf{c}=(c_n)_{n=1}^{\infty}$ be a sequence of positive numbers such that  $c_n<1/2$ for all $n$. The {\bf symmetric Cantor set} $E_{\mathbf{c}}$ is the Cantor set  whose generating intervals satisfy
\[
|I_{\omega 0}| = |I_{\omega 1}| = c_n |I_{\omega}| \quad \mbox{ for every $n\geq 1$ and  } \omega\in \Omega^{n-1}.
\]
\end{defn}

In what follows, we let $E_{\mathbf{c}}$ be the symmetric Cantor set associated with a sequence $\mathbf{c}=(c_n)_{n=1}^{\infty}$ of real numbers satisfying $0<c_n<1/2$. By the Cantor set construction, the gap intervals of $E_{\bf c}$ satisfy
\[
|G_{\omega} |= (1-2c_n)|I_{\omega}|
\]
 for each $\omega\in \Omega^{n-1}$. Moreover, it is standard to see that $E_{\bf c}$ admits the following additive representation:
\begin{equation}\label{eq:cantor}
	E_{\mathbf{c}}= \biggl\{ \sum_{n=1}^{\infty}\omega_n u_n \colon \omega_n\in\{ 0,1 \} \biggr\},
\end{equation}
where $u_1=1-c_1$ and $u_n=\left(\prod_{j=1}^{n-1} c_j\right )(1-c_n)$ for $n\geq 2$.

It is known (see~\cite{FeWeW97}) that if $\inf c_n >0$, then
\begin{equation}\label{eq:dim}
	\begin{aligned}
		\lbdim E_{\mathbf{c}} &=\hdim E_{\mathbf{c}}= \liminf_{n\to\infty} \frac{n\log 2}{-\log (c_1c_2\dotsm c_n)}, \\
		\ubdim E_{\mathbf{c}} &=\pdim E_{\mathbf{c}}= \limsup_{n\to\infty} \frac{n\log 2}{-\log (c_1c_2\dotsm c_n)},
	\end{aligned}
\end{equation}
where $\pdim$ denotes the packing dimension; see \cite{Falconer2014} for its definition.

Now we can apply Theorem~\ref{thm:bilip} to obtain a sufficient condition for  $E_{\mathbf{c}}$ to be the attractor of  a bi-Lipschitz IFS.
\begin{prop}\label{prop:bilip1}
	Let  $\mathbf{c}=(c_n)_{n\ge1}$ be a sequence such that $0<c_n<1/2$ for all $n\in \N$ and let $E_{\mathbf{c}}$  be the corresponding symmetric Cantor set,   with $\phi_0$, $\phi_1$ being defined in~\eqref{eq:IFSmap}, respectively. Then $\phi_0$ and~$\phi_1$ are bi-Lipschitz contractions on $[0,1]$ if and only if
\begin{equation}\label{n:biLipi1}
	0<\inf_{n\ge1}\frac{c_n(1-2c_{n+1})}{1-2c_n},\qquad \sup_{n\ge1}\frac{c_n(1-2c_{n+1})}{1-2c_n}<1.
\end{equation}	
\end{prop}

\begin{proof}
By a direct calculation, for each $\omega \in \Omega^n$ and $i\in \{0,1\}$, we have
\[
\frac{|G_{i\omega}|}{|G_{\omega}|} = \frac{c_n(1-2c_{n+1})}{1-2c_n} \quad \mbox{and} \quad \frac{|I_{i\omega}|}{|I_{\omega}|} = c_{n+1}.
\]
Notice that the condition $\inf c_n = 0$ implies that $  \displaystyle \inf \frac{c_n(1-2c_{n+1})}{1-2c_n} = 0$. Therefore, this proposition follows immediately from Theorem~\ref{thm:bilip}.
\end{proof}

We now construct a symmetric Cantor set, which is the attractor of a bi-Lipschitz IFS but its box dimension does not exist. 

\begin{prop}\label{p:symboxnotexist}
There exists a symmetric Cantor set generated by a bi-Lipschitz IFS and its box dimension does not exist.
\end{prop}

\begin{proof}
	Choose a sequence $\mathbf{c}=(c_n)_{n\ge1}$ with $c_n\in\{1/3,1/4\}$ such that
	\[ \lim_{n\to\infty}\frac{\log c_1 + \dots + \log c_n}{n} \quad\text{does not exist}. \]
	Let $E_{\bf c}$ be the symmetric Cantor set associated with ${\bf c}$. By~\eqref{eq:dim}, the box dimension of $E_{\mathbf{c}}$ does not exist. On the other hand, the sequence~$\mathbf{c}$ satisfies the condition~\eqref{n:biLipi1}. Let $\phi_0,\phi_1$ be  the maps given by~\eqref{eq:IFSmap}, in which $E$ is replaced by $E_{\bf c}$. Then~\eqref{e-Cantor} holds with $E$ replaced by $E_{\bf c}$. So by Proposition~\ref{prop:bilip1}, $\{\phi_0, \phi_1\}$ is a bi-Lipschitz IFS on $[0,1]$, having $E_{{\bf c}}$ as its attractor.
\end{proof}

\begin{rem}
\label{rem-Cantor}From the proof of Proposition~\ref{p:symboxnotexist} and the dimensional result \eqref{eq:dim} on symmetric Cantor sets, it follows easily  that for any given values $s,t\in (0,1]$ with $s\leq t$, there exists a symmetric Cantor set $E_{\bf c}$ such that $E_{\bf c}$ is the attractor of a bi-Lipschitz IFS, and moreover, $\lbdim E_{\bf c}=s$ and $\ubdim E_{\bf c}=t$.
\end{rem}

\subsection{Differentiability of the maps} The differentiability of the maps~$\phi_0$ and~$\phi_1$ associated with a given symmetric Cantor set $E_{\bf c}$ can be studied relatively easily. Notice that the two maps are of course differentiable on~$[0,1]\setminus E_{\mathbf{c}}$. However, we will show that they are non-differentiable at every point in $E_{\mathbf{c}}$ if the box dimension of $E_{\mathbf{c}}$ does not exist.

\begin{prop}\label{prop-no-diff}
	Let $E_{\mathbf{c}}$ be a symmetric Cantor set given by~\eqref{eq:cantor} and $\phi_0,\phi_1$ the maps given by~\eqref{eq:IFSmap}, in which $E$ is replaced by $E_{\bf c}$. If there exist $i\in\{0,1\}$ and $x\in E_{\mathbf{c}}$ such that $\phi_i$ is differentiable at~$x$, then the box dimension of $E_{\mathbf{c}}$ exists.
\end{prop}
\begin{proof}
Suppose that the derivative $\phi_i'(x)$ exists for some $i\in \{0,1\}$ and $x\in E_{\bf c}$.
	Since $x\in E_{\mathbf{c}}$, there exists $(\omega_n)_{n=1}^\infty\in \{0,1\}^\N$ such that
	\[ x = \omega_1(1-c_1)+\omega_2c_1(1-c_2)+\dots+\omega_nc_1c_2\dotsm c_{n-1}(1-c_n)+\cdots. \]
	For $n\in \N$, let $y_n\in E_{\mathbf{c}}$ be given as
	\[ y_n =\omega_1(1-c_1)+\omega_2c_1(1-c_2)+\dots+(1-\omega_n)c_1c_2\dotsm c_{n-1}(1-c_n)+\cdots. \]
	That is, $y_n$ is obtained from $x$ by replacing $\omega_n$ in the $n$th term in the summation defining $x$ with $(1-\omega_n)$.
	By the definition of~$\phi_i$ (see~\eqref{eq:IFSmap}), we have
	\[ \phi_i(x) = \pi (i0^{\infty})+ \omega_1c_1(1-c_2)+\dots+\omega_nc_1c_2\dots c_n(1-c_{n+1}) + \dotsb \]
	and
	\[ \phi_i(y_n) = \pi (i0^{\infty})+ \omega_1c_1(1-c_2)+\dots+(1-\omega_n)c_1c_2\dots c_n(1-c_{n+1}) + \dotsb. \]
	Consequently,
	\[ \frac{\phi_i(x)-\phi_i(y_n)}{x-y_n} = \frac{(2\omega_n-1)c_1\dotsm c_n(1-c_{n+1})}{(2\omega_n-1)c_1\dotsm c_{n-1}(1-c_n)} = \frac{c_n(1-c_{n+1})}{1-c_n}. \]
	Since $y_n$ converges to $x$, the existence of $\phi_i'(x)$  implies the existence of
	\[ \lim_{n\to\infty}\frac{c_n (1-c_{n+1})}{1-c_n}.\] Taking logarithms and averaging, it follows that the limit
	\[ \lim_{n\to\infty} \frac{1}{n}\sum_{j=1}^n\log \frac{c_j(1-c_{j+1})}{1-c_j} \]
	exists. But
	\[ \frac1{n}\sum_{j=1}^n\log \frac{c_j(1-c_{j+1})}{1-c_j}  = \frac1{n} \sum_{j=1}^n\log c_j + \frac{1}{n}\log \frac{1-c_{n+1}}{1-c_1}. \]
	Since $c_n\in (0,1/2)$ for all $n\ge1$, $\displaystyle\frac{1}{n}\log \frac{1-c_{n+1}}{1-c_1}\to 0$ as $n\to\infty$. It follows that the limit
	\[ \lim_{n\to\infty}\frac1{n}\sum_{j=1}^n\log c_j   \]
	exists. By~\eqref{eq:dim}, the box dimension of $E_{\mathbf{c}}$ exists.
\end{proof}

\subsection{Further results}\label{subsec:exactproof}
The following result is an application of the Shannon--McMillan--Breiman theorem. For the statement of  the theorem and the concept of measure-theoretic entropy of an invariant measure, the reader is referred to \cite{Walters}.
\begin{thm}\label{thm:symmetricnotexact}
    Consider the symmetric Cantor set $E_{\bf c}$ constructed in the proof of Proposition~\ref{p:symboxnotexist}. Let $\{\phi_0,\phi_1\}$ be the generating bi-Lipschitz IFS of $E_{\bf c}$, and let $s, t$ be the lower and upper  box dimensions of $E_{\bf c}$, respectively. Note that $0 < s < t \leq 1$. Let $\nu$ be an ergodic invariant measure for the left shift on $\Omega^{\N}$ with entropy $h>0$, and let $\mu = \pi_* \nu$ where $\pi$ is the coding map from~\eqref{e-pi}. Then at $\mu$-a.e.~$x \in \supp(\mu)$,
    \[
    \underline{\dim}(\mu,x) = \frac{hs}{\log 2}  <  \frac{ht}{\log 2} = \overline{\dim}(\mu,x).
    \]
\end{thm}
\begin{proof}
    Fix $\varepsilon > 0$. The Shannon--McMillan--Breiman theorem tells us that for $\mu$-a.e.~$x$, there exists $N_x \in \N$ such that for all $n \geq N_x$,
    \begin{equation}\label{eq:smbused}
    \mu(I_n(x)) \in \left[e^{-(h+\varepsilon)n},e^{-(h-\varepsilon)n}\right].
    \end{equation}

    By~\eqref{eq:dim}, for all $n$ sufficiently large,
    \[
    s-\varepsilon \leq \frac{n\log 2}{-\log (c_1 c_2\dotsm c_n)} \leq t + \varepsilon.
    \]
    This means that each of the $2^n$ intervals $I_{\omega}$ with $\omega \in \Omega^n$ has length
    \[
    c_1 c_2\dotsm c_n \in \left[2^{-n/(s-\varepsilon)}, 2^{-n/(t+\varepsilon)}\right].
    \]
    Since each $c_i$ is at most $1/3$, if $c_1 c_2\dotsm c_{n+1} < \delta \leq c_1 c_2\dotsm c_n$ then $$E_{\bf c} \cap I_{n+1}(x) \subseteq E_{\bf c}\cap B(x,\delta) \subseteq E_{\bf c}\cap I_n(x).$$
     Therefore by~\eqref{eq:smbused}, for $\mu$-a.e.~$x$,
    \[
    \frac{\log \mu(B(x,\delta))}{\log \delta} \leq \frac{(h+\varepsilon)(n+1)}{(n \log 2 )/ (t+\varepsilon)} = \frac{(h+\varepsilon)(t+\varepsilon)}{\log 2} + O\left(\frac1n\right).
    \]
    Similarly, for $\mu$-a.e.~$x$,
    \[
    \frac{\log \mu(B(x,\delta))}{\log \delta} \geq \frac{(h-\varepsilon)n}{((n+1) \log 2) / (s-\varepsilon)} = \frac{(h-\varepsilon)(s-\varepsilon)}{\log 2} +O\left(\frac1n\right).
    \]
    Thus the following holds for $\mu$-a.e.~$x$,
    \begin{equation}\label{eq:localdimbounds}
    \frac{(h-\varepsilon)(s-\varepsilon)}{\log 2} \leq \underline{\dim}(\mu,x) \leq \overline{\dim}(\mu,x) \leq \frac{(h+\varepsilon)(t+\varepsilon)}{\log 2}.
    \end{equation}

    Again by~\eqref{eq:dim}, there exist strictly increasing sequences of natural numbers $(n_k)_{k=1}^\infty$ and $(m_k)_{k=1}^\infty$ such that for all $k$,
    \[
    \frac{n_k \log 2}{-\log (c_1 c_2\dotsm c_{n_k})} \leq s+\varepsilon \quad \mbox{ and } \quad \frac{m_k \log 2}{-\log (c_1 c_2\dotsm c_{m_k})} \geq t-\varepsilon.
    \]
    Using the fact that each of the $c_i$ is at most $1/3$ and the Shannon--McMillan--Breiman theorem, it follows that for $\mu$-a.e.~$x$, there exists $N_{x}\in \N$ such that for all $n\geq N_{x}$ we have
    \[
    \mu(B(x,c_1 c_2\dotsm c_{n})) = \mu(I_n(x)) \in \left[e^{-(h-\varepsilon)n},e^{-(h+\varepsilon)n}\right].
    \]
    Combining this observation with the properties of the sequences $(n_{k})$ and $(m_{k})$ stated above, it follows that
    \[
    \underline{\dim}(\mu,x) \leq \frac{(h+\varepsilon)(s+\varepsilon)}{\log 2}\quad\text{ and }\quad\overline{\dim}(\mu,x) \geq \frac{(h-\varepsilon)(t-\varepsilon)}{\log 2}
    \]for $\mu$-a.e.~$x$. Combining these inequalities with~\eqref{eq:localdimbounds} and using that $\varepsilon$ was arbitrary, the proof is now complete.
\end{proof}

\begin{proof}[Proof of Theorem~\ref{thm:notexact}]
Here we do not need to use the Shannon--McMillan--Breiman theorem, because $\mu(I_{\omega}) = 2^{-n}$ for all $\omega \in \Omega^n$. The remainder of the proof proceeds exactly as in the proof of Theorem~\ref{thm:symmetricnotexact}. We emphasise that because of the fact $\mu(I_{\omega}) = 2^{-n}$ for all $\omega \in \Omega^n$ the estimates used in the proof of Theorem~\ref{thm:symmetricnotexact} hold for all $x$ instead of for $\mu$-a.e.~$x$.
\end{proof}

Equipped with Proposition~\ref{prop:bilip1} we can now also prove Theorem~\ref{thm:emptyinterior}.
\begin{proof}[Proof of Theorem \ref{thm:emptyinterior}]
Take ${\bf c}=(c_n)_{n\geq 1}$ with $c_n = \frac12-\frac{1}{2^{n+1}}$. One can check that~\eqref{n:biLipi1} is satisfied. By Proposition~\ref{prop:bilip1} we know that the maps $\phi_{0}$ and $\phi_{1}$ are bi-Lipschitz contractions. It is straightforward to observe that the attractor of this IFS has empty interior. Moreover, it follows from our construction that the Lebesgue measure of the attractor is equal to $\prod_{n=1}^{\infty}(1-\frac{1}{2^n})$ which is strictly positive.
\end{proof}

\section{Proofs of Corollaries \ref{cor:dynamical} and \ref{cor:dynamicalmeasure} }
\label{S-8}
In this section, we prove Corollaries \ref{cor:dynamical} and \ref{cor:dynamicalmeasure}.

\begin{proof}[Proof of Corollary~\ref{cor:dynamical}]
From the proof of Theorem~\ref{t:!bdim}, there exist two bi-Lipschitz contraction maps $\phi_0,\phi_1 \colon \R \to \R$ with $$0 = \phi_0(0) < \phi_0(1) < \phi_1(0) < \phi_1(1) = 1,$$
 and with the attractor $E \subset [0,1]$ of the IFS $\{\phi_0,\phi_1\}$ satisfying $$\dim_{\mathrm H} E < \lbdim E < \ubdim E < \adim E.$$
  Fix $\delta > 0$ small enough that $\phi_0(1) + 2\delta < \phi_1(0) - 2\delta$. We now define $f$ piece by piece. For $0 \leq x \leq \phi_0(1)$ let $f(x) = \phi_0^{-1}(x)$ and for $\phi_1(0) \leq x \leq 1$ let $f(x) = \phi_1^{-1}(x)$. On each of the four intervals \[
(-\infty,0], \ \ [\phi_0(1),\phi_0(1) + 2\delta], \ \ [\phi_1(0)-2\delta,\phi_1(0)], \ \ [1,\infty),
\]
let $f$ be affine with slope $2$. Finally, let $f$ be affine on $[\phi_0(1) + 2\delta , \phi_1(0)-2\delta]$ with the (negative) slope required to ensure $f$ is continuous. Now let $$U = (-\delta,\phi_0(1)+\delta) \cup (\phi_1(0)-\delta,1+\delta).$$
 Since $\phi_0$ and $\phi_1$ are contractions, the desired uniform local expansion for $f$ on $U$ holds. Finally, since $f$ maps $\phi_0(F)$ bijectively onto $F$, and $f$ maps $\phi_1(F)$ bijectively onto $F$, we have $E = f(E) = f^{-1}(E) \cap U$, as required.
\end{proof}

\begin{proof}[Proof of Corollary~\ref{cor:dynamicalmeasure}]
    The set $E$ can be taken as the symmetric Cantor set that is the attractor of the bi-Lipschitz IFS $\{\phi_0,\phi_1\}$ from Proposition~\ref{prop:bilip1}, so $\lbdim E < \ubdim E$. Let $\mu$ be the stationary measure for this IFS with probability weights $(1/2,1/2)$. Then,  by Theorem~\ref{thm:notexact}, we have $\underline{\dim}(\mu,x) = \lbdim E$ and $\ubdim E = \overline{\dim}(\mu,x)$ for all $x \in \supp(\mu)$. The construction of the dynamical system now proceeds exactly as in the proof of Corollary~\ref{cor:dynamical}. The fact that $\mu$ is $f$-invariant and ergodic follows because the resulting dynamical system is isomorphic to the full shift on two symbols equipped with the $(1/2,1/2)$ Bernoulli product measure.
\end{proof}

\subsection*{Acknowledgments}
The authors would like to thank Andr\'{a}s M\'{a}th\'{e} for bringing their attention to Question~\ref{q-1}, and Xiong Jin for useful discussions. Baker and Banaji were supported by Baker's EPSRC New Investigators Award (EP/W003880/1) at Loughborough University.
Banaji was also supported by the Marie Skłodowska-Curie Actions postdoctoral fellowship FoDeNoF (no. 101210409) from the European Union, and by Tuomas Orponen's Research Council of Finland grant (no. 355453), at the University of Jyväskylä.
Feng was supported by the General Research Fund grant (projects CUHK14303021, CUHK14305722) from the Hong Kong Research Grant Council, and by a direct grant for research from the Chinese University
of Hong Kong. Lai was supported by the AMS-Simons Research Enhancement Grants for Primarily Undergraduate Institution (PUI) Faculty. Xiong
was supported by NSFC grant 12271175 and 11871227.


\begin{thebibliography}{10}

\bibitem{BanRu24}
A.~Banaji and A.~Rutar.
\newblock Lower box dimension of infinitely generated self-conformal sets.
\newblock {\em Preprint},
 arXiv:2406.12821, 2024.

\bibitem{BaranyKaenmakiExact}
B.~B\'ar\'any and A. K\"aenm\"aki.
\newblock Ledrappier–Young formula and exact dimensionality of self-affine measures.
\newblock {\em Adv. Math.} 318:88--129, 2017.

 \bibitem{BSS2023}
B.~B\'{a}r\'{a}ny,  K.~Simon, and B.~ Solomyak.
\newblock  {\em Self-similar and Self-affine Sets and Measures}.
\newblock Math. Surveys Monogr. 276, American Mathematical Society, Providence, RI, 2023.

\bibitem{Barreira1996}
L.~M. Barreira.
\newblock A non-additive thermodynamic formalism and applications to dimension theory of hyperbolic
dynamical systems.
\newblock {\em Ergodic Theory Dynam. Systems}, 16(5):871--927, 1996.

 \bibitem{Bedford1984}
T.~Bedford.
\newblock Crinkly curves, Markov partitions and box dimensions in self-similar
  sets.
\newblock PhD Thesis, The University of Warwick, 1984.

\bibitem{BisPe17}
C.~J. Bishop and Y.~Peres.
\newblock {\em Fractals in probability and analysis}, volume 162 of {\em
  Cambridge Studies in Advanced Mathematics}.
\newblock Cambridge University Press, Cambridge, 2017.

\bibitem{Bo75}
R.~Bowen.
\newblock A horseshoe with positive measure.
\newblock {\em Invent. Math.}, 29:203--204, 1975.

\bibitem{CJPPS}
M.~Cs\"{o}rnyei, T.~Jordan, M.~Pollicott, D.~Preiss and B.~Solomyak. \newblock Positive-measure self-similar sets without interior. \newblock {\em Ergodic Theory Dynam. Systems}, 26(3):739--754, 2006.

\bibitem{Falco89}
K.~J. Falconer.
\newblock Dimensions and measures of quasi self-similar sets.
\newblock {\em Proc. Amer. Math. Soc.}, 106(2):543--554, 1989.

\bibitem{Falconer2014}
K.~J. Falconer.
\newblock {\em Fractal geometry---Mathematical foundations and applications}.
\newblock John Wiley \& Sons, Ltd., Chichester, third edition, 2014.

\bibitem{Falco97}
K.~J. Falconer.
\newblock {\em Techniques in fractal geometry}.
\newblock John Wiley \& Sons Ltd., Chichester, 1997.

\bibitem{FengExact}
D.-J. Feng.
\newblock Dimension of invariant measures for affine iterated function systems.
\newblock {\em Duke Math. J.}, 172(4):701--774, 2023.

\bibitem{FenHu09}
D.-J. Feng and H.~Hu.
\newblock Dimension theory of iterated function systems.
\newblock {\em Comm. Pure Appl. Math.}, 62(11):1435--1500, 2009.

\bibitem{FengWang2005}
D.-J. Feng and Y. Wang,
\newblock A class of self-affine sets and self-affine measures.
\newblock {\em J. Fourier Anal. Appl.}, 11(1):107--124, 2005.

\bibitem{FeWeW97}
D.-J. Feng, Z.~Wen, and J.~Wu.
\newblock Some dimensional results for homogeneous {M}oran sets.
\newblock {\em Sci. China Ser. A}, 40(5):475--482, 1997.

\bibitem{Fraser}
J.~M. Fraser.
\newblock Assouad dimension and fractal geometry.
\newblock Cambridge Tracts in Math., 222, Cambridge University Press, Cambridge, 2021.

\bibitem{FraserInterpolation}
J.~M. Fraser.
\newblock Interpolating between dimensions.
\newblock In Proceedings of Fractal Geometry and Stochastics VI (eds. U. Freiberg, B. Hambly, M. Hinz and S. Winter), volume 76, Birkhäuser, Progr. Probab., 2021.

\bibitem{GatPe97}
D.~Gatzouras and Y.~Peres.
\newblock Invariant measures of full dimension for some expanding maps.
\newblock {\em Ergodic Theory Dynam. Systems}, 17(1):147--167, 1997.

\bibitem{Hutch81}
J.~E. Hutchinson.
\newblock Fractals and self-similarity.
\newblock {\em Indiana Univ. Math. J.}, 30(5):713--747, 1981.

\bibitem{Jurga23}
N.~Jurga.
\newblock Nonexistence of the box dimension for dynamically invariant sets.
\newblock {\em Anal. PDE}, 16(10):2385--2399, 2023.

\bibitem{LiuWu2003}
Y.-Y. Liu and J. Wu.
\newblock Dimensions for random self-conformal sets.
\newblock {\em Math. Nachr.}, 250:71--81, 2003.

\bibitem{Mattila1995}
P.~Mattila.
\newblock {\em Geometry of sets and measures in {E}uclidean spaces}.
\newblock Cambridge University Press, Cambridge, 1995.

\bibitem{MU1996}
R.~D. Mauldin and M. Urba\'nski.
\newblock Dimensions and measures in infinite iterated
function systems.
\newblock {\em Proc. Lond. Math. Soc.}, 73:105--154, 1996.

\bibitem{MU1999}
R.~D. Mauldin and M. Urba\'nski.
\newblock Conformal iterated function systems with applications to the geometry of continued fractions.
\newblock {\em Trans. Amer. Math. Soc.}, 351:4995--5025, 1999.

\bibitem{McMul84}
C.~McMullen.
\newblock The {H}ausdorff dimension of general {S}ierpi\'nski carpets.
\newblock {\em Nagoya Math. J.}, 96:1--9, 1984.

\bibitem{PerSol}
Y.~Peres and B.~Solomyak.
\newblock Problems on self-similar sets and self-affine sets: an update. Fractal geometry and stochastics, II (Greifswald/Koserow, 1998), 95–106.
Progr. Probab., 46
Birkhäuser Verlag, Basel, 2000.

\bibitem{PesinWeiss}
Y. Pesin and H. Weiss.
\newblock On the dimension of deterministic and random Cantor-like sets, symbolic dynamics, and the Eckmann--Ruelle conjecture.
\newblock {\em Commun. Math. Phys.}, 182:105--153, 1996.

\bibitem{Walters} P.~Walters.
\newblock {\em An Introduction to {E}rgodic Theory}.
\newblock Springer, New York, 1982.
\end{thebibliography}
\end{document}